%% file: Main-pontryagin.tex
\documentclass[10pt]{amsart}
\usepackage{graphicx}
\usepackage{amssymb,amsmath,amsthm}
\usepackage{epstopdf}
\usepackage{mdframed}
\usepackage{empheq}
\usepackage{bm}
\usepackage{hyperref}
\usepackage{enumitem}
\usepackage{subcaption}
\usepackage[round]{natbib}
\setcitestyle{numbers,square}

\newtheorem{theorem}{Theorem}

\newtheorem{lemma}[theorem]{Lemma}
\newtheorem{proposition}[theorem]{Proposition}
\newtheorem{remark}[theorem]{Remark}

\newcommand\by{\underline y}
\newcommand\bz{\underline z}
\newcommand\bx{\underline x}

\newcommand\cF{\mathcal F}

\newcommand\cL{\mathcal L}

\newcommand\cP{\mathcal P}

\newcommand\cW{\mathcal W}

\newcommand\NN{\mathbb N}
\newcommand\RR{\mathbb R}

\newcommand{\x}{\bx}
\newcommand{\X}{\underline X}
\newcommand{\Y}{\underline Y}

\newcommand{\xdim}{\ell}
\newcommand{\ydim}{q}

\makeatletter
\renewcommand\l@paragraph[2]{}
\renewcommand\l@subparagraph[2]{}
\makeatother
\DeclareMathOperator*{\argmin}{arg\,min}

\usepackage{xcolor}
\definecolor{mypurple}{RGB}{140,0,255}
\definecolor{myred}{rgb}{255,0,0}

\definecolor{mydarkturquoise}{RGB}{0,206,209}
\definecolor{mydeeppink}{RGB}{255,20,147}
\definecolor{darkblue}{RGB}{0,0,140}
\definecolor{blue2}{RGB}{0,0,0}
\definecolor{middleblue}{RGB}{0,0,71}
\definecolor{light-gray}{gray}{0.9}

\definecolor{ProcessBlue}{cmyk}{1,0,0,0.40}
\definecolor{Black}{cmyk}{0,0,0,1}
\definecolor{Red}{cmyk}{0,1,1,0.2}
\definecolor{Green}{cmyk}{0.9,0,1,0}
\definecolor{Orange}{cmyk}{0,0.61,0.87,0.5}
\definecolor{Fuchsia}{cmyk}{0.47,0.91,0,0.06}
\definecolor{PineGreen}{cmyk}{0.92,0,0.59,0.30}

\title[Convergence of large population games to MFGs with interaction through the controls]{Convergence of large population games to mean field games with interaction through the controls}

\author{Mathieu Lauri\`ere  \& Ludovic Tangpi}

\thanks{mathieu.lauriere@nyu.edu, NYU Shanghai (work done while at Princeton University, ORFE); ludovic.tangpi@princeton.edu, Princeton University, ORFE. The work of M. Lauriere was supported by ARO grant AWD1005491 and NSF award AWD1005433. The work of L. Tangpi was supported by NSF grant DMS-2005832.}	
\date{\today}

\begin{document}

\maketitle

\begin{abstract}
This work considers stochastic differential games with a large number of players, whose costs and dynamics interact through the empirical distribution of both their states and their controls.   
We develop a new framework to prove convergence of finite-player games to the asymptotic mean field game. 
Our approach is based on the concept of propagation of chaos for forward and backward weakly interacting particles which we investigate by stochastic analysis methods, and which appear to be of independent interest.
These propagation of chaos arguments allow to derive moment and concentration bounds for the convergence of Nash equilibria.
\end{abstract}

\tableofcontents

\input{introduction}

\input{section1}

\input{section2}

\input{section3}

\input{section4}

\bibliographystyle{plainnat}
\bibliography{convergenceMFGC-bib,references-Concen_RM,references-MFGsupp.bib}

\end{document}

%% file: introduction.tex
\section{Introduction}
\label{sec:intro}

The motivation behind this paper is to present a systematic method to investigate the asymptotic behavior of a class of symmetric $N$-player stochastic differential games in continuous time as the number of players $N$ becomes large.
To be more precise, let us briefly describe such a game in the non-cooperative case.
We consider a game in which each player (or agent) $i \in \{1, \dots, N\}$ controls a diffusion process $X^{i,N}$ whose evolution is given by
\begin{equation*}
	dX^{i,N}_t = b\Big(t, X^{i,N}_t, \alpha^{i,N}_t, \frac1N\sum_{j=1}^N\delta_{(X^{j,N}_t, \alpha^{j,N}_t)}  \Big)\,dt + \sigma \,dW^i_t
\end{equation*}
for some independent Brownian motions $W^1, \dots, W^N$ where $\alpha^{i,N}$ is a control process chosen by player $i$ and $\delta_x$ is the Dirac delta mass at $x$.
The measurability of $\alpha^{i,N}$ will be precised below.
Agent $i$ tries to minimize an individual cost
\begin{equation}
\label{eq:cost}
		J(\alpha^i; \underline \alpha^{-i}):=	E\Big[\int_0^Tf\Big(t, X_t^{i,N}, \alpha^{i,N}_t, \frac1N\sum_{j=1}^N\delta_{(X^{j,N}_t, \alpha^{j,N}_t)} \Big)\,dt + g\big(X^{i,N}_T, \frac1N\sum_{j=1}^N\delta_{X^{j,N}_T} \big) \Big]
\end{equation}
where we denote $\underline \alpha^{-i}:= (\alpha^1,\dots,\alpha^{i-1}, \alpha^{i+1},\dots, \alpha^N)$.
In this context, it is natural to investigate the concept of \emph{Nash equilibrium} $(\hat\alpha^{1,N},\dots, \hat\alpha^{N,N})$.
See ~\S~\ref{subsec:model} for definitions and a more precise description of the model.
Unfortunately, as the number of players becomes large, the $N$-Nash equilibrium becomes analytically and (especially) numerically intractable.
The groundbreaking idea of Lasry $\&$ Lions \cite{MR2295621} and \citet{MR2346927} is to argue that, heuristically, for such a symmetric game, when $N$ goes to infinity, $\hat\alpha^{i,N}$ should converge to a so-called \emph{mean field equilibrium} $\hat\alpha^i$, which is defined as follows. 
For a fixed (measurable) measure flow $(\xi_t)_{t \ge 0}$ with second marginals $(\mu_t)_{t\ge0}$ let $(\hat\alpha^\xi_t)_{t \ge 0}$ be a solution of the stochastic control problem
\begin{equation*}
	\begin{cases}
	\displaystyle
		\inf_{\alpha}E\Big[\int_0^Tf(t, X^\alpha_t, \alpha_t,\xi_t)\,dt + g(X^\alpha_T, \mu_T)\Big]\\
		dX^\alpha_t = b(t, X^\alpha_t, \alpha_t, \xi_t)\,dt + \sigma\,dW_t^i.
	\end{cases}
\end{equation*}
A flow of measures $\hat\xi$ is an equilibrium flow if it satisfies the following \emph{consistency condition}: the law of $(X_t^{\hat\alpha^{\hat\xi}}, \hat\alpha^{\hat\xi}_t)$ equals $\hat\xi_t$ for every $t\in [0,T]$; the associated control $\hat\alpha^i$ is an equilibrium control. 
The question at the heart of the present paper is to know \emph{how far $\hat\alpha^i$ is from $\hat\alpha^{i,N}$}. 
In other terms, we are interested in an estimation of the ``error'' $|\hat\alpha^{i,N} - \hat\alpha^i|$.

It is only after more than a decade of intensive research on mean field games that the intriguing heuristics mentioned above have been put into rigorous mathematical ground and in satisfactory generality. Notably, the works of \citet{MR3520014} and \citet{Fischer17} proved \emph{convergence} results on the $N$-Nash equilibria to the mean field equilibrium as $N$ goes to \emph{infinity} for open-loop controls. Using a PDE on the Wasserstein space called the master equation,  \citet{carda15} proved convergence for closed-loop controls, even in the presence of common noise. 
We also refer to works by \citet{LackerAAP}, \citet{Del-Lac-Ram_Concent,Del-Lac-Ram-19}, \citet{Carda_AMO17} for more recent progress on this convergence question.
Anticipating our brief discussion of these papers in the soon-to-come literature review (see~\S~\ref{subsec:literature}), let us mention at this  point that with the exception of \cite{Del-Lac-Ram_Concent}, none of the above cited papers investigates non-asymptotic results, nor do their settings cover games with \emph{interactions through the distribution of controls} (or ``control interactions'' for short).

Games with control interactions, sometimes called ``extended'', occur when the dynamics or the cost function of player $i$ may explicitly depend on the empirical distribution of the \emph{controls} of the other players, and not just on their respective states.
Such games were first introduced by Gomes et. al. \cite{MR3160525} and their investigation quickly picked-up momentum due to their relevance in various problems e.g. in economics and finance.
References are provided below (see~\S~\ref{subsec:literature}).
One important aspect of our analysis will be to include the treatment of such games.

\subsection{Main results: informal statements and method}
The main result of this paper is to show that (even) for games with interactions through the controls, under sufficient regularity and convexity assumptions on the coefficients of the game one obtains a non-asymptotic estimate of the ``error'' term $E[|\hat\alpha^{i,N}_t - \hat\alpha_t^i|^2]$ and consequently convergence of $\hat\alpha^{i,N}$ to $\hat\alpha^i$. 
This moment estimate is bolstered by concentration inequalities (some of which dimension-free) notably bounding the probability that the Wasserstein distance between the empirical measure of the $N$-Nash equilibrium and the law of the mean-field equilibrium exceeds a given threshold.
The price to pay for these non-asymptotic bounds is to require either small enough time horizon or additional monotonicity conditions on the coefficients.
The contribution of this article is also methodological.
In fact, we design a three-step approach to bound the error:
\begin{itemize}
	\item[(i)] Characterize the solution of the $N$-player game by a system of forward-backward stochastic differential equations (FBSDE).
	\item[(ii)] Investigate asymptotic properties of the system of equations, showing in particular that it converges to a McKean-Vlasov FBSDE (see definition below).
	\item[(iii)]Show that the limiting McKean-Vlasov FBSDE characterizes the mean field equilibrium.
\end{itemize}
To achieve step (ii), we further develop the theory of \emph{backward propagation of chaos} initiated by the authors in \cite{backward-chaos}.
The idea here is that, roughly speaking, the FBSDEs characterizing the $N$-player game can be interpreted (themselves) as a system of weakly interacting particles evolving forward and backward in time.
A substantial part of the article is devoted to the investigation of non-asymptotic, strong propagation of chaos type results for such particle systems.
At the purely probabilistic level, these results extend the original ideas of Sznitman \cite{MR1108185} introduced for interacting (forward) particles to fully coupled systems of interacting forward and backward particles.
Due to the independent relevance of these convergence results, this part of the paper is presented in a self-contained manner and so that it can be read separately.  
In fact, in this article, aside from the (non-cooperative) large population games discussed so far, we illustrate applications of this ``forward-backward propagation of chaos'' by proving 
convergence of a system of second order parabolic partial differential equations written on an Euclidean space to a so-called \emph{master equation}, a second order PDE written on the Wasserstein space.
This allows for convergence results to PDEs on infinite dimensional spaces similar to the ones derived by \mbox{\citet{carda15}}, with different types of nonlinearities. 

\subsection{Literature review}
\label{subsec:literature}
The investigation of the limit theory in large population games started with the works of Lasry $\&$ Lions \cite{MR2269875,MR2295621} further extended by \citet{MR3127148}, Bardi $\&$ Priuli \cite{Bardi-Priuli} and \citet{Gomes-Mohr-Sou13}.
These papers share the limitations of treating either problems with linear coefficients or assuming that agents have controls which are not allowed to depend on other players' states.
In the breakthrough works of \citet{MR3520014} and \citet{Fischer17}, 
the authors prove rather general convergence results for the \emph{empirical measure} of the states of the agents at equilibrium  using probabilistic techniques.
We also refer to \citet{LackerAAP} for interesting further developments, notably for the case of \emph{closed-loop} controls.
The analyses of these authors use the notion of relaxed controls and study associated controlled martingale problems.
This technique seems hard to extend to games with control interactions considered here, and it provides compactness results rather than convergence rates.
However, one central advantage of this approach is that it does not assume uniqueness of the mean field equilibrium, which we do (at least in our main theorem).
This shortcoming is shared with the PDE-based approaches of \citet{Del-Lac-Ram_Concent,Del-Lac-Ram-19} and \citet{cardaliaguet2019master} (but some of these works additionally need existence and bounds on the first and second order derivatives of the solution of the associated master equation). 
In fact, our approach is related to these methods in that they both rely on optimality conditions characterizing the equilibrium. However, instead of using optimality conditions phrased in terms of PDEs, we use FBSDEs characterizations. 
As a result, the technique developed here is a purely probabilistic one and we do not restrict ourselves to Markovian controls as in the PDE approaches.

Beyond its methodological aspects, our paper contributes to the large population game and the mean field game literature by its analysis of games with control interactions.
Mean field games with such interactions are sometimes referred to as ``extended MFG'' or ``MFG of controls'' and have been introduced by~\citet{MR3160525,MR3489048}. 
Interaction through the controls' distribution is particularly relevant in economics and finance, see e.g.~\cite{MR3359708,MR3755719,MR3805247} and~\cite[Section 3.3.1]{MR3195844} (see also~\cite[Sections 1.3.2 and 4.7.1]{MR3752669}).  
Some aspects of the PDE approach and the probabilistic approach to such games have been treated respectively in~\cite{MR3941633,BonnansHadikhanlooPfeiffer2019,Kobeissi2019} and in~\cite{MR3325272}. 
Note also that this paper focuses on \emph{open--loop} equilibria.
The convergence problem for \emph{closed--loop} equilibria is considered by \cite{LackerAAP,cardaliaguet2019master} using very different methods.
Furthermore, let us finally point out that the method developed in this paper also apply to non-cooperative games and the results have natural PDE interpretation.
These connections are presented in details in the ArXiv version of the paper \cite{pontryagin21}.

\subsection{Organization of the paper}
In the next section we present the probabilistic setting and formally state our main results pertaining to the convergence of the $N$-Nash equilibrium to the mean field equilibrium.
The emphasis is put on non-asymptotic results and concentration estimates.
Section \ref{sec:pontryagin} is dedicated to the discussion of versions of Pontryagin's maximum principle for games with interaction through the controls. %
The investigation of propagation of chaos for forward-backward interacting particles is carried out in Section \ref{sec:fbsde chaos}.
These elements are put together in Section \ref{sec:limits} to prove the main results stated in Section \ref{sec:main results}.

%% file: section1.tex
\section{Main results: formal statements}
\label{sec:main results}
Let $T>0$ and $d \in \mathbb{N}$ be fixed, and denote by  $(\Omega, \mathcal{F},P)$ a probabilitty space carrying a sequence of independent $\mathbb{R}^d$-valued Brownian motions $(W^i)_{i\in \mathbb{N}}$.
For every positive integer $N$, let $W^1, \dots, W^N$ be $N$ independent copies of $W$ and $\mathcal{ F}_0$ be an initial $\sigma$-field independent of $W^1,\dots, W^N$.
We equip $\Omega$ with the filtration $\mathbb{F}^N:=(\mathcal{F}_t^N)_{t\in[0,T]}$, which is the completion of the filtration generated by $W^1, \dots, W^N$ and $\mathcal{F}_0$, and we further denote by $\mathbb{F}^i:=(\mathcal{F}_t^i)_{t\in[0,T]}$, which is the completion of the filtration generated by $W^i$ and $\mathcal{F}_0$. 
Without further mention, we will always use the identifications
\begin{equation*}
	W \equiv W^1 \quad \text{and}\quad \cF \equiv \cF^1.
\end{equation*}

Given a vector $\underline x := (x^1, \dots, x^N)\in (\mathbb{R}^n)^N$, for any $n \in \mathbb{N}$, denote by 
\begin{equation*}
	L^N(\x) := \frac 1N \sum_{j=1}^N\delta_{x^j}
\end{equation*}
the empirical measures associated to $\underline x$.
It is clear that $L^N(\x)$ 
belongs to $\cP_p(\mathbb{R}^n)$, the set of probability measures on $\mathbb{R}^n$ with finite $p$-moments.
Given a random variable $X$, we denote by
\begin{equation*}
	\cL(X) \quad \text{the law of $X$ with respect to $P$}.
\end{equation*}
Throughout the paper, $C$ denotes a generic strictly positive constant.
In the computations, the constant $C$ can change from line to line, but this will not always be mentioned. 
\emph{However, $C$ will never depend on $N$.}

Let us now formally state the main results of this work.
\subsection{The $N$-player game}
\label{subsec:model}
We consider an $N$-agent game where player $i$ chooses an admissible strategy $\alpha^i$ to control her state process, which has dynamics
\begin{equation}
\label{eq:N-SDE}
	dX^{i, \underline \alpha}_t = b(t,X^{i, \underline \alpha}_t, \alpha^i_t, L^{N}(\underline X_t^{\underline\alpha}, \underline \alpha_t)) dt + \sigma dW^i_t, \quad X^{i, \underline\alpha}_0\sim \mu^{(0)},
\end{equation}
for some function $b$, a matrix $\sigma$ and a distribution $\mu^{(0)} \in \cP_2(\mathbb{R}^\xdim)$, where the state depends on an average of the states and controls of all the \emph{players} through the empirical measure $L^{N}(\underline X^{\underline\alpha}_t, \underline \alpha_t)$. The initial states $X^{i, \underline\alpha}_0$ are assumed to be i.i.d. 
Let $m \in \NN$ and let $\mathbb{A}\subseteq \mathbb{R}^m$ be a closed convex set.
The set of admissible strategies is defined as\footnote{Unless otherwise stated, we denote by $|\cdot|$ the Euclidean norm and by $ab:=a\cdot b $ the inner product, regardless of the dimension of the Euclidean space.
}
	\begin{equation*}
		\mathcal{A}:= \left\{\alpha:[0,T]\times \Omega\to \mathbb{A}\,\,\, \mathbb{F}^N\text{--progressive such that }  E\Big[\int_0^T|\alpha_t|^2\,dt \Big]<+\infty \right\}.
	\end{equation*}
Given two functions $f$ and $g$, the cost that agent $i$ seeks to minimize, when the strategy profile is $\underline \alpha = (\alpha^1,\dots,\alpha^N)$, is $	J(\alpha^i; \underline \alpha^{-i})$ given in \eqref{eq:cost}.
Note that under our assumptions (specified below) the cost $J$ is well defined for all admissible strategy profiles.
As usual, one is interested in constructing a Nash equilibrium $\hat{\underline \alpha}:= (\hat\alpha^{i},\dots, \hat\alpha^{n})$, 
that is, admissible strategies $(\hat\alpha^{1},\dots,\hat\alpha^N)$ such that for every $i=1,\dots,N$ and $\alpha \in \mathcal{A}$ it holds that
\begin{equation*}
 	J^i(\hat{\underline\alpha}) \le J(\alpha^i;\hat{\underline\alpha}^{-i}).
\end{equation*}

When such a Nash equilibrium exists for every $N$, our aim is to investigate its asymptotic properties as $N \to \infty$.
In particular, we give (regularity) conditions on the coefficients of the diffusions and the cost under which the Nash equilibrium of the $N$-player game converges to the mean-field equilibrium which we define below. 
We denote by $\mathcal{W}_2(\xi, \xi')$ the second order Wasserstein distance between two probability measures $\xi,\xi'$ and by $\partial_\xi h, \partial_\mu h$ and $\partial_\nu h$ the so-called $L$-derivatives of a function $h$ in the variable of the probability measure $\xi \in \mathcal{P}_2(\mathbb{R}^{\xdim}\times \mathbb{R}^m)$, $\mu \in \mathcal{P}_2(\mathbb{R}^{\xdim})$ and $\nu \in \mathcal{P}_2(\mathbb{R}^{m})$, respectively. 
See e.g.~\cite{MR2401600,PLLcollege} or~\cite[Chapter 5]{MR3752669} for definition and further details.

We will use the following assumptions, on which we comment after stating our main results, see Remark~\ref{rem:assumptions}.
\begin{enumerate}[label = (\textsc{A1}), leftmargin = 30pt]
	\item The function $b:[0,T]\times \mathbb{R}^\xdim\times \mathbb{R}^m\times \mathcal{P}_2(\mathbb{R}^\xdim\times \mathbb{R}^m) \to \mathbb{R}^\xdim$ is continuously differentiable in its last three arguments and satisfies the Lipschitz continuity and linear growth conditions
	\begin{equation*}
	\begin{cases}
		|b(t,x,a,\xi) - b(t,x^\prime, a^\prime, \xi^\prime)| \le L_f\big( |x-x'| +|a - a^\prime| + \mathcal{W}_2(\xi,\xi') \big) 
		\\
		|b(t,x,a,\xi)| \le C\left(1+|x| +|a| + \Big(\int_{\mathbb{R}^{\xdim+m}}|v|^2\,\xi(dv) \Big)^{1/2} \right)
    \end{cases}
	\end{equation*}
	for some $C,L_f>0$ and all $x, x' \in \mathbb{R}^\xdim$, $a,a' \in \mathbb{R}^m$, $t \in [0,T]$ and $\xi,\xi' \in \cP_2(\mathbb{R}^\xdim\times \mathbb{R}^m)$.

	The functions $f:[0,T]\times \mathbb{R}^\xdim\times \mathbb{R}^m\times \mathcal{P}(\mathbb{R}^\xdim\times \mathbb{R}^m) \to \mathbb{R}$ and $g:\mathbb{R}^\xdim\times \mathcal{P}(\mathbb{R}^\xdim) \to \mathbb{R}$ are continuously differentiable ($f$ in its last three arguments) and of quadratic growth:
	\begin{equation*}
	\begin{cases}
		|f(t,x,a,\xi)| \le C\Big(1+|x|^2 + |a|^2 + \int_{\mathbb{R}^{\xdim+m}}|v|^2\,\xi(dv)\Big)
		\\
		|g(x,\mu)| \le C\Big(1 +|x|^2 + \int_{\mathbb{R}^{\xdim}}|v|^2\,\mu(dv) \Big)
	\end{cases}
	\end{equation*}
	for some $C>0$ and all $x \in \mathbb{R}^\xdim$, $a \in \mathbb{R}^m$, $t \in [0,T]$ and $\xi \in \cP_2(\mathbb{R}^\xdim\times \mathbb{R}^m)$.
\label{a1}
\end{enumerate}
\begin{enumerate}[label = (\textsc{A2}), leftmargin = 30pt]
	\item The functions $b$ and $f$ can be decomposed as  
	\begin{equation}
	\label{eq:decom bf}
	\begin{cases}
		b(t,x,a,\xi) := b_1(t,x, a, \mu) + b_2(t, x, \xi) 
		\\
		f(t,x,a,\xi) = f_1(t,x,a,\mu) + f_2(t,x,\xi)
	\end{cases}
	\end{equation}
	for some functions $b_1,b_2$, $f_1$ and $f_2$, where $\mu$ is the first marginal of $\xi$. 
	\label{a2}
\end{enumerate}
\begin{enumerate}[label = (\textsc{A3}), leftmargin = 30pt]
	\item 
	Considering the function 
	\begin{equation}
	\label{eq:def-H-hyp}
		H(t,x,y,a,\xi)= f(t,x,a,\xi) + b(t,x,a,\xi)y,
	\end{equation}
	 there is $\gamma>0$ such that
	\begin{equation}
	\label{eq:strong convex}
		H(t,x,y,a,\xi) - H(t,x,y,a',\xi) - (a - a')\partial_aH(t,x,y,a,\xi) \ge\gamma |a-a'|^2
	\end{equation}
	for all $a,a' \in \mathbb{A}$,  $x \in \mathbb{R}^\xdim$, $a \in \mathbb{R}^m$, $t \in [0,T]$ and $\xi \in \cP_2(\mathbb{R}^\xdim\times \mathbb{R}^m)$;
	and the functions $x\mapsto g(x, \mu)$ and $(x,a)\mapsto H(t, x, y,a, \xi)$ are convex, where $\mu$ is the first marginal of $\xi$.
	In addition, the functions 
	\begin{equation*}
		\text{$\partial_a H(t,\cdot, \cdot,\cdot,\cdot)$, $\partial_xH(t,\cdot, \cdot,\cdot,\cdot)$ and $\partial_xg(\cdot,\cdot)$ are $L_f$-Lipschitz--continuous}
	 \end{equation*}
	  and of linear growth:
	\begin{equation*}
		\begin{cases}
			|\partial_xH(t, x,a,y,\xi)| \le C\Big(1 + |x|+ |y| + \big(\int_{\mathbb{R}^{\xdim}}|v|^2\mu(dv)\big)^{1/2} \Big)\\
			|\partial_aH(t, x,a,y,\xi)| \le C\Big(1 + |x|+ |a| + |y| + \big(\int_{\mathbb{R}^{\xdim+\ydim}}|v|^2\xi(dv)\big)^{1/2} \Big)\\
			|\partial_xg(x,\mu) \le C\Big(1 +|x| + (\int_{\mathbb{R}^{\xdim}}|v|^2\xi(dv))^{1/2} \Big)
		\end{cases}
	\end{equation*} 
	for some $C>0$ and all $x \in \mathbb{R}^\xdim$, $a \in \mathbb{R}^m$, $t \in [0,T]$ and $\xi\in \cP_2(\mathbb{R}^\xdim\times \mathbb{R}^m)$	where $\mu$ is the first marginal of $\xi$.
	\label{a3}
\end{enumerate}
\begin{enumerate}[label = (\textsc{A4}), leftmargin = 30pt]
	\item For every $(t,x,a,\xi)\in [0,T]\times \mathbb{R}^\xdim\times\mathbb{A}\times \mathcal{P}_2(\mathbb{R}^{\xdim\times m})$ and $(u,v) \in \mathbb{R}^\xdim\times\mathbb{R}^m$ we have 
	\begin{equation*}
	\begin{cases}
		| \partial_\mu b(t,x,a,\mu)(u)| \le C \\
		 | \partial_\xi f(t,x,a,\xi)(u,v) |\le C\Big(1 + |u| + |x| + \Big(\int_{\mathbb{R}^{\xdim}}|v|^2\mu(dv) \Big)^{1/2} \Big)\\
		| \partial_\mu g(x,\mu)(u)| \le C\Big(1 + |u|  + |x| +\Big(\int_{\mathbb{R}^\xdim}|v|^2\mu(dv) \Big)^{1/2} \Big)
		\end{cases}
	\end{equation*}
	for some $C>0$ and all $x \in \mathbb{R}^\xdim$, $a \in \mathbb{R}^m$, $t \in [0,T]$ and $\xi\in \cP_2(\mathbb{R}^\xdim\times \mathbb{R}^m)$	where $\mu$ is the first marginal of $\xi$.
\label{a4}
\end{enumerate}	
\begin{enumerate}[label = (\textsc{A5}), leftmargin = 30pt]
	\item The matrix $\sigma$ is uniformly elliptic. 
	That is, there is a constant $c>0$ such that 	$\langle \sigma\sigma'x,x\rangle \ge c|x|^2$  for every $x \in \mathbb{R}^{\xdim}$.
	\label{a5}
\end{enumerate}

\subsection{The Mean field game}
The mean field game that corresponds to the above $N$-player game is described as follows:
Given a flow of distributions $(\xi_t)_{t \in [0,T]}$ with $\xi_t \in \cP_2(\mathbb{R}^\xdim\times \mathbb{R}^m)$ with first marginal $\mu_t\in \cP_2(\mathbb{R}^\xdim)$, the cost of an infinitesimal agent is
$$
	J^\xi(\alpha) = E\bigg[\int_0^T f(t,X^\alpha_t, \alpha_t,\xi_t) dt + g(X^\alpha_T,\mu_T) \bigg]
$$
with the dynamics
$$
	dX^\alpha_t = b(t,X^\alpha_t, \alpha_t, \xi_t) dt + \sigma dW_t^i, \qquad X^\alpha_0 \sim \mu^{(0)}.
$$
The admissibility set on which the cost function $J^\xi$ is minimized is
\begin{equation*}
		\mathfrak{A}:= \bigg\{\alpha:[0,T]\times \Omega\to \mathbb{A}\,\,\, \mathbb{F}^i\text{--progressive such that }  E\Big[\int_0^T|\alpha_t|^2\,dt \Big]<+\infty \bigg\}.
	\end{equation*} 
The goal for the agent is to find $\hat\alpha^\xi \in \mathfrak{A}$ minimizing $J^\xi$ and satisfying the fixed point (or consistency) condition 
\begin{equation*}
	\xi_t = \cL(X^{\hat\alpha^\xi}_t, \hat\alpha^\xi_t)\quad \text{for all } t.
\end{equation*}
The first main result of this paper is the following link between the $N$-player game and the (asymptotic) MFG in small time horizon:
\begin{theorem}
\label{thm:main limit} 
	Let conditions \ref{a1}-\ref{a5} be satisfied and assume that there is $k>2$ such that $\mu^{(0)}$ admits moments of order $k$.
	Assume that the $N$-player game admits a Nash equilibrium $\hat{\underline\alpha}^{N} \in \mathcal{A}^N$. 
	Then, there is $\delta>0$ such that if $T\le \delta$, for each $i$ the sequence $(\hat\alpha^{i,N})_N$ converges to a mean field equilibrium $\hat\alpha^i \in \mathfrak{A}$ in the sense that it holds
	\begin{equation*}
		E\Big[|\hat{\alpha}^{i,N}_t - \hat\alpha_t^i|^2\Big] \le C (r_{N,m+\xdim, k} + r_{N,\xdim,k} )
	\end{equation*}
	for all $t \in[0,T]$ and $N\in \mathbb{N}$ and some constant $C>0$	where, for any $M,N,k $ we put $r_{N,M,k}:= r_{N, M,k,2}$ and
	\begin{equation}
	\label{eq:def_r-nmqp}
		r_{N,M,k,p}:=\begin{cases}
			N^{-1/2} + N^{-(k-p)/k}, 	&\text{if } p>M/2 \text{ and } k\neq 2p\\
			N^{-1/2}\log(1+N) + N^{-(k-p)/k}, 	&\text{if }  p=M/2 \text{ and } k \neq 2p\\
			N^{-2/M} + N^{-(k-p)/k}, 	&\text{if } M>2p \text{ and } k\neq M/(M-p).
		\end{cases}
	\end{equation}
\end{theorem}
In the case of linear quadratic games the convergence rate can be simplified to the \emph{optimal rate} $O(1/N)$.
The proof of this statement can be found in the ArXiv version of the paper, see \cite[Theorem 11]{pontryagin21}.
The small time assumption of Theorem \ref{thm:main limit} can be replaced by a monotonicity property on the drift.
\begin{enumerate}[label = (\textsc{M}), leftmargin = 30pt]
		\item The Hamiltonian admits a minimizer 
		\begin{equation*}
		 	\Lambda(t,x, y, \mu) \in \argmin_{a \in \mathbb{A}}H(t, x,a,y, \xi)
		\end{equation*}
		where $\mu$ is the first marginal of $\xi$.
		The drift $b$ satisfies the monotonicity condition
		\begin{equation}
		\label{eq:mon.con.b}
				(x-x')\cdot\Big(b(t,x,\Lambda(t,x,y,\mu),\xi)-b(t,x',\Lambda(t,x',y,\mu),\xi) \Big)\le -K_b|x-x'|^2
		\end{equation}
		for all $x,x' \in \mathbb{R}^\ell$, $(t, \xi)\in [0,T]\times \mathcal{P}(\mathbb{R}^\ell\times \mathbb{R}^m)\to \mathbb{R}^\ell$ and some constant $K_b>0$.
		Moreover, 
		it holds
		\begin{equation}
		\label{eq:mon.con.b.H.g}
			\begin{cases}
				(y-y')\cdot\Big( b(t,x,\Lambda(t, x, y, \mu),\xi)-b(t,x,\Lambda(t, x, y',\mu),\xi) \Big) \le -K|y-y'|^2\\
				(x - x')\cdot\Big( \partial_xH(t,x',\Lambda(t, x', y, \mu),\xi)-\partial_xH(t,x,\Lambda(t, x, y,\mu),\xi) \Big) \le -K|x-x'|^2\\
				(x -x')\cdot\Big(\partial_xg(x, \mu) - \partial_xg(x',\mu)\Big)\ge K|x-x'|^2
			\end{cases}
		\end{equation}
		for all $t \in [0,T]$, $x,x', y, y' \in \RR^\xdim$ and $\xi \in \cP_2(\RR^{\xdim\times m})$, and for some constant $K>0$.
		\label{M}
	\end{enumerate}
In the statement and proof of the next result it will be judicious to distinguish the Lipschitz constant of $b$ in each of its arguments, so that the Lipschitz--continuity condition in \ref{a1} now reads
	\begin{enumerate}[label = (\textsc{A1}'), leftmargin = 30pt]
	\item The function $b$ satisfies
	\begin{equation*}
		\begin{cases}
		|b(t,x,a,\xi) - b(t,x',a',\xi')|\le L_{b,x}|x-x'| + L_{b,a}|a-a'| + L_{b,\xi}\cW_2(\xi,\xi')\\
		|b(t,x,a,\xi)|\le C(1 + |x| + |a| + \int_{\mathbb{R}^\ell}|x|^2\xi(d x))\end{cases}
	\end{equation*}
		for some constants $C>0$, $L_{b,x}, L_{b,a}, L_{b,\xi}>0$ and all $x,x' \in \mathbb{R}^\xdim$, $a,a' \in \mathbb{R}^m$, $t \in [0,T]$ and $\xi,\xi'\in \cP_2(\mathbb{R}^\xdim\times \mathbb{R}^m)$	where $\mu$ is the first marginal of $\xi$.
\label{a1prime}
	\end{enumerate}
	Recall that $L_f$ is the Lipschitz--constant of $\partial_xH$ and $\partial_xg$ as stated in \ref{a3}.
	With these notation, put
	\begin{equation*}
		\begin{cases}
			K_1 := 2\Big(\frac{3}{2\gamma}L_{b,a}L_f + L_{b,\xi} \Big)\Big( 8Te^{7L_fT} L_f + \frac12 \Big) + L_{b,\xi} + \frac{L_{b,a}L_f}{2\gamma} \\
			K_2 := 16TL_f^2(L_f + T) \Big( L_{b,\xi} +  \frac{3L_fL_{b,a}}{2\gamma} \Big)^2e^{2L_f(6+L_f)} + L_{b,\xi}+ \frac{L_{b,a}L_f}{2\gamma}\\
			K_{3,k}:=4(k-1)\frac{L_{b,\xi} + L_fL_{b,a}}{2^k\gamma(L_F + L_G)\exp(kTL_F(2+\frac{2L_F}{k(k-1)}))}.
		\end{cases}
	\end{equation*}
	\begin{theorem}
	\label{thm:main.limit.arbitrary.time} 
		Let conditions \ref{a1prime},\ref{a2}-\ref{a5} and \ref{M} be satisfied and assume that there is $k>2$ such that $\mu^{(0)}$ admits moments of order $k$.
		Suppose that the $N$-player game admits a Nash equilibrium $\hat{\underline\alpha}^{N} \in \mathcal{A}^N$.
		If the constant $K_b$ satisfies
		\begin{equation*}
		 K_b> \max(K_1,K_2,K_{3,k}),
	\end{equation*}
	then for each $i$ the sequence $(\hat\alpha^{i,N})_{N}$ converges to a mean field equilibrium $\hat\alpha^i \in \mathfrak{A}$ in the sense that it holds
		\begin{equation*}
			E\Big[|\hat{\alpha}^{i,N}_t - \hat\alpha_t^i|^2 \Big] \le C (r_{N,m+\xdim, k} + r_{N,\xdim,k} )
		\end{equation*}
		for all $t \in[0,T]$ and $N\in \mathbb{N}$ and some constant $C>0$.
	\end{theorem}
	\begin{remark}
		In  Assumption \ref{M}, the fact that $\Lambda$ depends only upon the first marginal of $\xi$ is due to \ref{a2}.
		This will be proved below.
		Moreover, The reader will observe in the proof that the essential condition needed to derive the convergence is \eqref{eq:mon.con.b}.
		The conditions \eqref{eq:mon.con.b.H.g} are needed to guarantee existence (for $T$ arbitrary) of the McKean--Vlasov FBSDE \eqref{eq:MkV.in.proof} characterizing the game.
		The conditions \eqref{eq:mon.con.b.H.g} can be dropped when this equation admits a unique solution.
	\end{remark}
We now complement Theorems \ref{thm:main limit} and \ref{thm:main.limit.arbitrary.time}  with concentration estimates for the $N$-Nash equilibrium.
\begin{theorem}
\label{thm:concen Nash}
	Under the conditions of Theorem \ref{thm:main limit}, it holds that\footnote{We recall that $\underline{\hat\alpha}^N:=(\hat\alpha^1,\dots,\hat\alpha^N)$, where $\hat\alpha^i$ is a mean field equilibrium of the mean field game with Brownian motion $W^i$.}
	\begin{equation*}
		E\big[\cW_2(L^N(\underline{\hat\alpha}_t^N), \cL(\hat\alpha_t))\big] \le C\big( r_{N,2\xdim,k} + r_{N,\xdim,k}\big)
	\end{equation*}
 	for all $(t,N) \in [0,T]\times\mathbb{N}$.

 	If in addition $\mu^{(0)}$ is a Dirac mass
	then there is a constant $c(L_f)$ depending only on the Lipschitz constants of $b,f,g$ and $\partial_xH$ such that if $T\le c(L_f)$, then for every $N\ge 1$ and $\varepsilon>0$ it holds that
	\begin{equation}
	\label{eq:deviation for alpha}
		P(h(\underline{\hat\alpha}^N_t) - E[h(\underline{\hat\alpha}^N_t)] \ge \varepsilon) \le \frac{C}{N\varepsilon^2}+e^{-K\varepsilon^2}
	\end{equation}
	 for two given constants $K,C$ which do not depend on $N$ and for every $1$-Lipschitz function $h:\mathbb{R}^{mN}\to \mathbb{R}$.
	In particular, for $N$ large enough it holds that 
	\begin{equation}
	\label{eq:concen Wasses}
	P\Big(\mathcal{W}_2\big(L^N(\hat{\underline\alpha}^{N}_t), \cL(\hat\alpha_t)\big)\ge \varepsilon\Big) \le \frac{C}{\varepsilon^2N^2} + e^{-KN\varepsilon^2}. 
	\end{equation}

	If the functions $b,f$ and $g$ satisfy \ref{M} and are such that
	\begin{equation}
	\label{eq:bound.on.H.for.T.large}
	 \begin{cases}
	 |\partial_aH(t, x, a,\xi)|+ |\partial_xH(t, x, a, \xi)| \le C\big(1 + |a| + (\int_{\mathbb{R}^m}|v|^2\,\nu(dv)^{1/2} \big)\\
	  |\partial_xg(x,\mu)|\le C
	 \end{cases}
	 \end{equation}
	 where $\nu$ is the second marginal of $\xi$, then for every $T>0$ and for $N$ large enough it holds
	 \begin{equation}
		P\Big(\int_0^Th(\underline{\hat\alpha}^N_t) - E[h(\underline{\hat\alpha}^N_t)]\,dt \ge \varepsilon\Big) \le \frac{C}{N\varepsilon^2}+e^{-K\varepsilon^2}
	\end{equation}
	 for two given constants $K,C$ which do not depend on $N$ and for every $1$-Lipschitz function $h:\mathbb{R}^{mN}\to \mathbb{R}$.
\end{theorem}
Before going any further, let us make a few remarks concerning our assumptions.

\begin{remark}
\label{rem:assumptions}
	Let us now briefly comment on the assumptions made in Theorems \ref{thm:main limit} and \ref{thm:concen Nash}.
	In a nutshell, both theorems tell us that under sufficient regularity and integrability of the coefficients of the game, we have convergence with explicit convergence rates.
	Condition \ref{a1}, \ref{a3} and \ref{a4} speak to these regularity and integrability conditions.
	These conditions, along with the convexity property \eqref{eq:strong convex} are typically assumed in the literature, even to guarantee solvability see e.g. \cite{Hamadene98,Ben-Fre84}.

	The conditions in \ref{a2} are structural conditions on the coefficients.
	These conditions are probably not essential from a mathematical standpoint.
	They are due to our method, which consists in finding an explicit representation of the equilibrium in terms of processes whose convergence can be derived, see  
	\eqref{eq:def alpha star}. 
	Thus, the conditions in \ref{a2} can be replaced by any other conditions ensuring such representations of the equilibria.
	Importantly \eqref{eq:decom bf} is not needed when we do not have mean-field interaction through the controls, but only through the states.
	The Lipschitz assumptions on $\partial_xH$ and $\partial_aH$ in \ref{a3} are not necessary when $\partial_xg$ is bounded and $\partial_xf$ and $\partial_xb$ are bounded in $x$.
	In fact, in this case, BSDE estimates show that the function $\partial_xH$ can be restricted to bounded $y$'s, so that these Lipschitz continuity conditions are automatically satisfied if $\partial_xb$ is Lipschitz. 
\end{remark}

%% file: section2.tex
\section{Pontryagin's maximum principle} 
\label{sec:pontryagin}
As explained in the introduction, two elements of our three-step approach to derive the limit consist in applying Pontryagin's maximum principles for $N$-agent games and for mean field games.
This section is dedicated to the presentation of these results.
In the case of $N$-agent games we give the ``necessary part'' of the maximum principle.
Since the case of mean field games is less involved, we present both the ``necessary'' and the ``sufficient'' parts.

\subsection{Pontryagin's maximum principle for $N$-agent games}
The goal of this section is to discuss Pontryagin's maximum principle of the $N$-agent game and derive characterization properties for the Nash equilibria.
Hereafter, for each $p\ge1$ and $k\in \mathbb{N}$ we denote
\begin{align*}
		{\mathcal S}^{p}(\mathbb{R}^{k}) := \left\{ Y \in \mathcal{H}^{0}(\mathbb{R}^{k}) \,\Big|\, E\Big[ \sup_{0 \le t \le T} |Y_t|^p\Big] < +\infty \right\}
		\\
		{\mathcal H}^{p}(\mathbb{R}^{k}) := \left\{ Z \in \mathcal{H}^{0}(\mathbb{R}^{k}) \,\Big|\, E \Big[\Big(\int_0^T |Z_t|^2 dt\Big)^{p/2}\Big] < +\infty \right\},
\end{align*}
		with ${\mathcal H}^{0}(\mathbb{R}^{k})$ being the space of all $\RR^k$-valued progressively measurable processes.
\begin{proposition}
\label{thm:N pontryagin}
	Let the conditions \ref{a1}, \ref{a4} and \ref{a5} be satisfied.
	If $\underline{\hat\alpha}$ is a Nash equilibrium of the $N$-player game, then for any admissible control $\underline \beta = (\beta^1,\dots, \beta^N)$ it holds
	\begin{equation}
	\label{eq:optim cond Hi}
		\partial_{\alpha^i}H^{N,i}(t,\underline X^{\underline{\hat \alpha}}_t, \underline{\hat\alpha}_t, \underline Y^{i,\cdot}_t)(\beta^i_t - \hat\alpha^i_t)\ge0 \quad P\otimes dt\text{-a.s., for all } i,
	\end{equation}
	where $H^{N,i}$ is the $i$-player's Hamiltonian given by 
	\begin{equation*}
		H^{N, i}(t,\underline x, \underline \alpha, \underline y) := f\left(t, x^i, \alpha^i, L^N(\underline x, \underline\alpha) \right) + \sum_{j =1}^Nb\left(t,x^j, \alpha^j,L^N(\underline x, \underline {\alpha})  \right) y^{i,j}
	\end{equation*}
	and putting $g^{N,i}(\underline x):= g(x^i,L^N(\underline x))$ and $\underline Y^{i,\cdot} = (Y^{i,1},\dots, Y^{i,N})$, $(Y^{i,j}, Z^{i,j,k})_{i,j,k}$ solves the adjoint equation
	\begin{equation}
	\label{eq:adjoint-fctN}
	d Y^{i,j}_t = - \partial_{x^j} H^{N,i}(t,\underline X^{\underline{\hat \alpha}}_t, \underline {\hat\alpha}_t, \underline Y^{i,\cdot}_t) dt + \sum_{k=1}^N Z^{i,j,k}_t dW^{k}_t,\quad Y^{i,j}_T = \partial_{x^j}g^{N,i}(\underline X^{\underline{\hat \alpha}}_T).
	\end{equation}
\end{proposition}
Note that $(Y^{i,j}, Z^{i,j,k})_{i,j,k}$ implicitly depend upon $\underline{\hat \alpha}$ but we omit to write this dependence to alleviate the notation. 
\begin{proof}
	If $\underline{\hat\alpha}$ is a Nash equilibrium, then player $i$ solves the stochastic control problem $\sup_{\alpha \in \mathcal{A}}J(\alpha, \underline{\hat\alpha}^{-i})$.
	That is, it holds
	\begin{equation*}
		J(\hat\alpha^i, \underline{\hat\alpha}^{-i}) =  \sup_{\alpha \in \mathcal{A}}J(\alpha, \underline{\hat\alpha}^{-i}).
	\end{equation*}
	Therefore, the result follows by application of the (standard) stochastic maximum principle, see e.g. \cite[Theorem 2.15]{MR3753660}.
\end{proof}

For later reference and for convenience of the reader, we spell-out the adjoint equations~\eqref{eq:adjoint-fctN} in terms of the functions $f,b,g$ appearing in the game. 
From \cite[Proposition 5.35]{MR3753660}, we have
$$
	\partial_{x^j} g^{N,i}(\underline x, \underline \alpha) 
	= \delta_{i,j} \partial_{x} g\left(x^i, L^N(\underline x) \right)
	+ \frac{1}{N} \partial_\mu g\left(x^i, L^N(\underline x) \right)(x^j),
$$
where $\delta_{i,j}=1$ if and only if $i=j$ and $0$ otherwise. 
Similar relations hold for $f$ and $b$, and for the partial derivatives with respect to the control variables. We deduce that
\begin{equation}
\label{eq:adjoint eq full terminal}
	Y^{i,j}_T 
	= \delta_{i,j} \partial_{x} g\left(X^{i,\underline{\hat\alpha}}_T,L^N(\underline X^{\underline{\hat\alpha}}_T) \right) 
	+  \frac{1}{N} \partial_\mu g\left(X^{i,\underline{\hat\alpha}}_T, L^N(\underline X^{\underline{\hat\alpha}}_T)\right) (X^{j,\underline{\hat\alpha}}_T),
\end{equation}
and
\begin{align}
\notag
	d Y^{i,j}_t 
	&= - \partial_{x^j} H^{N,i}(t, \underline X^{\underline{\hat \alpha}}_t, \underline {\hat\alpha}_t, \underline Y^{i,\cdot}_t) dt + \sum_{k=1}^N Z^{i,j,k}_t dW^{k}_t
	\\\notag
	&= - \Big( \delta_{i,j} \partial_{x} f\Big(t,X^{i,\underline{\hat\alpha}}_t, \hat\alpha^i_t, L^N(\underline X_t^{\underline{\hat \alpha}}, {\underline{\hat \alpha}}_t) \Big)
	+ \frac{1}{N} \partial_\mu f\Big(t, X^{i,\underline{\hat\alpha}}_t, \hat\alpha^i_t, L^N(\underline X_t^{\underline{\hat \alpha}}, {\underline{\hat \alpha}}_t) \Big)(X^{j,\underline{\hat\alpha}}_t,\hat\alpha^j_t) \Big)dt
	\\\notag
	&
	\qquad - \Big( \partial_{x} b\Big(t, X^{j,\underline{\hat\alpha}}_t, \hat\alpha^j_t, L^N(\underline X_t^{\underline{\hat \alpha}}, {\underline{\hat \alpha}}_t)\Big) Y^{i,j}_t\,dt	\\\label{eq:adjoint eq full}
	&\qquad + E_{(\tilde X, \tilde \alpha, \tilde Y) \sim \overline{\zeta}^{N,i}_t} \Big[ \partial_\mu b\Big(t,\tilde X_t, \tilde \alpha_t, L^N(\underline X_t^{\underline{\hat \alpha}}, {\underline{\hat \alpha}}_t)  \Big)(X^{j,\underline{\hat\alpha}}_t,\hat\alpha^j_t) \tilde Y_t \Big] \,dt+ \sum_{k=1}^N Z^{i,j,k}_t dW^{k}_t
\end{align}
where we used the notation $\overline{\zeta}^{N,i}_t := \frac{1}{N} \sum_{j=1}^N \delta_{(X^{j,\underline{\hat\alpha}}_t, \hat\alpha^j_t, Y^{i,j}_t)}$ for the empirical distribution of the triple $(X^{j,\underline{\hat\alpha}}_t,\hat \alpha^j_t, Y^{i,j}_t)_j$.

\subsection{Pontryagin's maximum principle for mean field games of controls}
\label{sec:mfg pontryagin}

Let us recall that the Hamiltonian $H$ is defined by~\eqref{eq:def-H-hyp}, i.e.
$$
	H(t,x, \alpha,y, \xi) = f(t,x, \alpha, \xi) + b(t,x, \alpha, \xi)  y.
$$
Recall the following optimality conditions for mean field games:
\begin{proposition}
\label{prop:Pontryagin-MFG}
If $\hat\alpha$ is a mean field equilibrium such that the mapping $t\mapsto \xi^{\hat\alpha}_t := \cL(X^{\hat\alpha}_t,\hat\alpha_t)$ is bounded and Borel measurable, then it holds that
	\begin{equation}
	\label{eq:opt-H-MFG}
		H(t,X^{\hat\alpha}_t, \hat\alpha_t, Y^{\hat\alpha}_t, \xi_t^{\hat\alpha}) = \inf_{a \in \mathbb{A}}	H(t,X^{\hat\alpha}_t, a, Y^{\hat\alpha}_t, \xi_t^{\hat\alpha}) \quad P\otimes dt\text{-a.s.}
	\end{equation}
	with $(X^{\hat\alpha}_t, Y^{\hat\alpha}_t, Z^{\hat\alpha}_t, \hat\alpha_t)$ solving the FBSDE system
	\begin{equation}
	\label{eq:adjoint process mfg}
	\begin{cases}
		dX^{\hat\alpha}_t = b(t,X^{\hat\alpha}_t, \hat\alpha_t, \xi^{\hat\alpha}_t) dt + \sigma dW_t, 
		&X^{\hat\alpha}_0 \sim \mu^{(0)},
		\\
		dY^{\hat\alpha}_t = -\partial_x H(t,X^{\hat\alpha}_t, \hat\alpha_t, Y^{\hat\alpha}_t, \xi^{\hat\alpha}_t) dt + Z^{\hat\alpha}_t d W_t, 
		&Y^{\hat\alpha}_T = \partial_x g(X^{\hat\alpha}_T, \cL(X^{\hat\alpha}_T)).
	\end{cases}
	\end{equation}

	Reciprocally, let $\hat \alpha$ be an admissible control with associated controlled process $X^{\hat\alpha}$ and adjoint processes $(Y^{\hat\alpha}, Z^{\hat\alpha})$ as given by \eqref{eq:adjoint process mfg}.
	Assume $t\mapsto \xi^{\hat\alpha}_t = \cL(X^{\hat\alpha}_t,\hat\alpha_t)$ is Borel--measurable and bounded (i.e. the second moment is bounded uniformly in $t$). Assume that for each $\xi \in \cP(\mathbb{R}^\xdim\times \mathbb{R}^m)$ with first marginal $\mu$ the functions $x\mapsto g(x, \mu)$ and $(x,a)\mapsto H(t, x,a, y , \xi)$ are $ dt$-a.s. convex and that $\hat\alpha$ satisfies \eqref{eq:opt-H-MFG}. 
	Then $\hat\alpha$ is a mean field equilibrium.
\end{proposition}
This result is standard, it follows for instance by application of  \cite[Theorems 2.15  and 2.16]{MR3752669} to the (standard) control problem parameterized by a given flow of measures, then use the consistency condition.

%% file: section3.tex
\section{Quantitative propagation of chaos for coupled FBSDE systems}
\label{sec:fbsde chaos}
This section studies abstract propagation of chaos type results for forward-backward systems of SDEs.
These results will be central for the proofs of the main theorems, but seem to be of independent interest.
Therefore, we present the section so that it can be read independently.

The main idea is that we consider a system of ``particles'' evolving forward and backward in time and with interactions through their empirical distributions.
We show that under mild regularity conditions on the coefficients of the equations describing the dynamics of the equations, the whole system converges to a system of McKean-Vlasov FBSDEs.
Moreover, we derive explicit convergence rates and concentration inequality results.
Propagation of chaos-type results for  backward SDEs (not coupled to forward systems) have been previously derived in \cite{Buck-Dje-Li-Peng09,Hu-Ren-Yang,backward-chaos}.

Let $d,\xdim,q\in \mathbb{N}$, we fix three functions 
\begin{align*}
  	&B:[0,T]\times\mathbb{R}^\xdim\times\mathbb{R}^q\times \cP_2(\mathbb{R}^\xdim\times \mathbb{R}^q)\to \mathbb{R}^\xdim,\\
  	&
  	F:[0,T]\times\mathbb{R}^\xdim\times \mathbb{R}^q\times \mathbb{R}^{q\times d}\times \cP_2(\mathbb{R}^\xdim\times\mathbb{R}^q)\to \mathbb{R}^q, \quad G:\mathbb{R}^\xdim \times \cP_2(\mathbb{R}^\xdim) \to \mathbb{R}^q
\end{align*}
and an $\ell \times d$ matrix $\sigma$ for some $\ell, d, q \in \mathbb{N}$.
Consider the coupled systems of FBSDEs
\begin{equation}
\label{eq: n fbsde}
	\begin{cases}
		\displaystyle 
		X^{i,N}_t = x^i_0 + \int_0^tB_u(X^{i,N}_u, Y^{i,N}_u, L^N(\X_u,\Y_u))\,du + \sigma\,W^i_t\\
		\displaystyle 
		Y^{i,N}_t = G(X^{i,N}_T, L^N(\underline X_T)) + \int_t^T F_u(X^{i,N}_u, Y^{i,N}_u, Z^{i,i,N}_u, L^N(\underline X_u, \underline Y_u)) \,du\\\quad\qquad - \sum_{k=1}^N\int_t^TZ^{i,k,N}_u\,dW^k_u,
	\end{cases}
\end{equation}
with $i=1,\dots,N$, and for given i.i.d., $\mathcal{F}_0$-measurable random variables $x^1_0, \dots, x^N_0$  with values in $\mathbb{R}^\xdim$, and where as above, we used the notation $\underline Y := (Y^1,\dots, Y^N)$ and $\underline X:= (X^1,\dots, X^N)$.
We recall that $W^1, \dots, W^N$ are independent $d$-dimensional Brownian motions.  
We will use the following conditions:
\begin{enumerate}[label = (\textsc{B1}), leftmargin = 30pt]
	\item The functions $B$, $F$ and $G$ are Lipschitz continuous, that is there are positive constants $L_B,L_F,L_G>0$ such that 
	\begin{align}
	\label{B1}
	\begin{cases}
		|F_t(x,y,z,\xi)-F_t(x',y',z',\xi')|\le L_F\left(|x-x'|+|y-y'| + |z-z'|+ \mathcal{W}_2(\xi, \xi')  \right)\\
		|B_t(x,y,\xi)-B_t(x',y',\xi')|\le L_B\left(|x-x'|+|y-y'| +  \mathcal{W}_2(\xi, \xi') \right)\\
		|G(x,\mu)-G(x',\mu')|\le L_G\left(|x-x'| + \mathcal{W}_2(\mu, \mu') \right)
		\end{cases}
	\end{align} 
	for every $t \in [0,T]$, $x, x' \in \mathbb{R}^\xdim$, $y,y' \in \mathbb{R}^q$, $z,z' \in \mathbb{R}^{q\times d}$ $\xi,\xi' \in \mathcal{P}_2(\mathbb{R}^\xdim\times \mathbb{R}^q)$ and $\mu, \mu' \in \mathcal{P}_2(\mathbb{R}^\xdim)$ .
	\label{b1}
\end{enumerate}
\begin{enumerate}[label = (\textsc{B2}), leftmargin = 30pt]
\item The functions $B ,F$ and $G$ satisfy the linear growth conditions
\begin{align*}
\label{B2}
	\begin{cases}
		|B_t(x,y,\xi)|\le L_B\left(1 + |x| + |y|+ \big(\int|v|^2\,\xi(dv) \big)^{1/2}\right)\\
 		|F_t(x,y,z,\xi)|\le L_F\left(1+ |x|+|y| + |z| + \big(\int|v|^2\,d\xi(v) \big)^{1/2}\right)\\
		|G(x,\mu)|\le L_G\left(1+ |x|+ \big(\int|v|^2\,d\mu(v)\big)^{1/2}  \right) .
	\end{cases}
 \end{align*} 
 \label{b2}
\end{enumerate}
\begin{enumerate}[label = (\textsc{B2}'), leftmargin = 30pt]
\item The functions $B ,F$ and $G$ satisfy the linear growth conditions 
\begin{align*}
	\begin{cases}
		|B_t(x,y,\xi)|\le L_B\left(1 + |y| + \big(\int|v|^2\,d\nu(v) \big)^{1/2}\right)\\
 		|F_t(x,y,z,\xi)|\le L_F\left(1+|y| + \big(\int|v|^2\,d\nu(v) \big)^{1/2}\right)\\
		|G(x,\mu)|\le L_G
	\end{cases}
 \end{align*} 
 where $\nu$ is the second marginal of $\xi$.
 \label{b2prime}
\end{enumerate}

\begin{remark}
\label{rem:system exists}
	Under the conditions \ref{b1}-\ref{b2} and \ref{a5}, it can be checked (see e.g. \cite[Remark 2.1]{backward-chaos}) that the functions
	\begin{equation*}
	\begin{cases}
		(\bx,\by)\mapsto (B_t(x^1, y^1, L^N(\bx,\by)), \dots, B_t(x^N, y^N, L^N(\bx,\by)))\\
		(\bx,\by,\bz)\mapsto (F_t(x^1, y^1,z^1 L^N(\bx,\by)), \dots, F_t(x^N, y^N, z^N,L^N(\bx,\by)))\\
		\x\mapsto (G(x^1, L^N(\x)), \dots, G(x^N,L^N(\bx)))
	\end{cases}
	\end{equation*}
	are Lipschitz continuous and of linear growth (with Lipschitz constant independent of $N$). 
	Thus, the unique solvability of the system \eqref{eq: n fbsde} when the time horizon $T$ is small enough is guaranteed e.g. by \cite[Theorem 4.2]{MR3752669}.
	Existence of a unique solution on arbitrary large time intervals typically requires additional conditions, for instance, if one additionally assumes \ref{b2prime}, see \cite[Theorem 4.1]{Ma-Pro-Yong} (when the coefficients are also smooth) or under monotonicity-type conditions on the drift and the generator for instance as assumed in \ref{b3} below, see \cite[Theorem 2.6]{Delarue} or \cite{Peng-Wu99}. 
\end{remark}
The first main result of this section is the following:
\begin{theorem}
\label{thm:mom bound syst}
	Assume that the conditions \ref{b1}-\ref{b2}, \ref{a5} are satisfied and that there is $k>2$ such that $E[|x^1_0|^k]<\infty$.
	Denote by $(\underline X, \underline Y, \underline Z ) \in \mathcal{S}^{2}(\mathbb{R}^{\xdim N})\times \mathcal{S}^{2}(\mathbb{R}^{qN})\times \mathcal{H}^{2}((\mathbb{R}^{q\times d})^{N\times N})$ the solution of the FBSDE \eqref{eq: n fbsde}.
	There is $\delta>0$ such that if $T \le \delta$ and
	the McKean-Vlasov FBSDE
	\begin{equation}
	\label{eq:MkVFBSDE}
	\begin{cases}
		\displaystyle 
		X_t = x^1_0 + \int_0^tB_u(X_u, Y_u, \cL({X_u,Y_u}))\,du + \sigma\,W_t\\
		\displaystyle 
		Y_t = G(X_T,\cL(X_T)) +\int_t^T F_u(X_u, Y_u, Z_u,\cL(X_u,Y_u))\,du - \int_t^TZ_u\,dW_u
	\end{cases}
	\end{equation}
	admits a unique solution $(X, Y, Z) \in \mathcal{S}^{2}(\mathbb{R}^\xdim)\times \mathcal{S}^{2}(\mathbb{R}^q)\times \mathcal{H}^{2}(\mathbb{R}^{q\times d})$, then it holds 
	\begin{equation}
	\label{eq:mom bound y system}
		\sup_{t \in [0,T]}E\Big[\mathcal{W}_2^2\big(L^N(\X_t, \Y_t),\cL(X_t,Y_t)\big) \Big] \le C\left( r_{N,q+\xdim,k} + r_{N,\xdim,k} \right) 
	\end{equation}
	for all $(t,N)\in[0,T]\times\mathbb{N}$, where $r_{N,q+\xdim,k}:=r_{N,q+\xdim,k,2}$ is given by \eqref{eq:def_r-nmqp}, and for some constants $C$ depending on $L_B,L_F, L_G$, $k$, $\sigma$, $E[|x^1_0|^k]$ and $T$.
	In addition for all $N\in\mathbb{N}$ we also have
	\begin{align}\notag
		E\bigg[\sup_{s\in [0,T]}|X^{1,N}_s - X^1_s |^2 \bigg]&+ E\left[|Y^{1,N}_t - Y^1_t |^2 \right]+ E\bigg[ \int_0^T|Z^{1,1,N}_s - Z_s^1|^2\,ds \bigg] \\
		&
		\le C\Big(r_{N,q+\xdim, k} + r_{N,\xdim,k} \Big) .
			\label{eq:process conv}
	\end{align}
\end{theorem}

\subsection{Proof of Theorem \ref{thm:mom bound syst}}
The arguments of the proof of Theorem \ref{thm:mom bound syst} are broken up into intermediate results that we present in this subsection.
Given a progressive $d$-dimensional process $\gamma$, we use the shorthand notation $\mathcal{E}_{s,t}(\gamma\cdot W)$ for the stochastic exponential of $\gamma$. That is, we put
\begin{equation*}
	\mathcal{E}_{s,t}(\gamma\cdot W) := \exp\Big(\int_s^t\gamma_u\,dW_u - \frac12\int_s^t|\gamma_u|^2\,du \Big).
\end{equation*}
In this whole subsection, we assume that~\eqref{eq:MkVFBSDE} admits a unique solution denoted by $(X,Y,Z)$. 
We start by proving useful moment bounds for solutions of McKean-Vlasov FBSDEs.
For simplicity in this subsection, we will put $L_f:= \max(L_B , L_F, L_G)$.
\begin{lemma}
\label{lem: moment d+5 system}
	Assume that the condition \ref{b2} is satisfied and that \eqref{eq:MkVFBSDE} admits a unique solution $(X,Y,Z) \in \mathcal{S}^2(\mathbb{R}^\xdim)\times \mathcal{S}^2(\mathbb{R}^\ydim)\times\mathcal{H}(\mathbb{R}^{\ydim\times d})$.
	Further assume that there is $k\ge2$ such that $E[|x^1_0|^k]<\infty$.
	If either $T$ is small enough or \ref{b3} is satisfied for $K_B$ therein such that 
\begin{equation}
\label{eq:choiceK_b.k}
	K_B\ge4(k-1)\frac{L_{B,y,\xi}}{2^k(L_F + L_G)\exp(kTL_F(2+\frac{2L_F}{k(k-1)}))}
\end{equation}
 where $L_{B,y,\xi}$ is the Lipschitz constant of $B$ in $(y,\xi)$, then it holds that
	\begin{equation}
	\label{eq:k-moment.X.Y}
		E\Big[ \sup_{t \in [0,T]} |X_t|^k \Big]+\sup_{t \in [0,T]} E\left[ |Y_t|^k \right]<\infty.
 	\end{equation}
\end{lemma}
\begin{proof}
	When $T$ is small enough, the proof follows standard FBSDE estimations. It is therefore omitted.

	Let us assume the the monotonicity condition \ref{b3} is satisfied.
	Applying It\^o's formula to $|X|^k$, using \ref{b3} and \ref{b2} yields 
	\begin{align*}
		|X_t|^k &\le |x^1_0|^k + k\int_0^t-K_B|X_u|^k + L_{B,y,\xi}|X_u|^{k-1}(1+|Y_u| + E[|X_u|^2]^{1/2} + E[|Y_u|^2]^{1/2})\,du\\
		&\quad + k\int_0^tX_u^{k-1}\sigma\,dW_u\\
		&\le |x^1_0|^k + k\int_0^t\Big(4(k-1)\frac{L_{B,y,\xi}}{\varepsilon}-K_B\Big)|X_u|^k + \varepsilon L_{B,y,\xi}\Big\{1+|Y_u|^{k} + E[|X_u|^2]^{k/2} + E[|Y_u|^2]^{k/2}\Big\}\,du\\
		&\quad + k\int_0^tX_u^{k-1}\sigma\,dW_u
	\end{align*}
	where $L_{B,y,\xi}$ denotes the Lipschitz contant of $B$ in $y$ and $\xi$, and where we used the inequality $xy \le x^p/p\varepsilon + \varepsilon y^q/q$ with $p,q$ H\"older conjugates.
	Thus, taking expectation (up to localization) and applying Gronwall's inequality
	\begin{align}
	\label{eq.estim.x.lemma}
		E[|X_t|^k] \le \varepsilon kL_{B,y,\xi}e^{(\frac{4(k-1)L_{B,y,\xi}}{\varepsilon} - K_B)T}E\bigg[\int_0^T|Y_u|^k\,du\bigg] + C.
	\end{align}
	Similarly, applying It\^o's formula to $Y^k$ and then Young's inequality for some $\eta>0$ yields
	\begin{align}
	\notag
		|Y_t|^k &\le E\bigg[|G(X_T,\mathcal{L}(X_T))|^k + L_Fk\int_t^T|Y_u|^{k-1}(|X_u| + |Y_u| + |Z_u| + E[|X_u|^2]^{1/2}+ E[|Y_u|^2]^{1/2})\\\notag
				&- \frac{k(k-1)}{2}\int_t^TY^{k-2}_u|Z_u|^2\,du \mid \mathcal{F}_t\bigg]\\\notag
				&\le E\bigg[2^kL_G(|X_T|^k + E[|X_T|^2]^{k/2} + 1) + kL_F(1 + \frac{k-1}{k} + \frac{1}{\eta})\int_t^T|Y_u|^k\,du\\
		\label{eq:estim.yy.lemma}
				& + L_F\int_t^T|X_u|^k + E[|X_u|^2]^{k/2}+ E[|Y_u|^2]^{k/2} + k(\eta L_F -\frac{k(k-1)}{2})\int_t^T|Y_u|^{k-2}|Z_u|^2\,du\mid \mathcal{F}_t\bigg]
	\end{align}
	where the second inequality uses \eqref{eq.estim.x.lemma}.
	Choosing $\eta$ such that $\eta L_f -\frac{k(k-1)}{2}=0$, taking expectation of both sides and applying Gronwall's inequality yields
	\begin{align*}
		E[|Y_t|^k] &\le C_1E\bigg[ |X_T|^k + \int_0^T|X_u|^k\,du \bigg] + C_2\\
		&\le C_1\varepsilon e^{(\frac{4(k-1)L_{B,y,\xi}}{\varepsilon} -K_B)T}TE\bigg[\int_0^T|Y_u|^k\,du\bigg] + C_2
	\end{align*}
	with $C_1 := 2^k(L_F + L_G)\exp(kTL_F(2+\frac{2L_F}{k(k-1)}))$.
	First choosing $\varepsilon>0$ small enough that $\varepsilon<[2^k(L_F + L_G)\exp(kTL_F(2+\frac{2L_F}{k(k-1)}))]^{-1}$ and then $K_B\ge 4(k-1)L_{B,y,\xi}/\varepsilon$, and integrating on both sides yields
	$E[\int_0^T|Y_u|^k\,du]<\infty$.
	In view of \eqref{eq.estim.x.lemma} and \eqref{eq:estim.yy.lemma} this yields the result.
\end{proof}

The proof of Theorem \ref{thm:mom bound syst} is based on the coupling technique used in \cite{backward-chaos}.
To this end, 
we fix $N$ i.i.d. copies $(\tilde X^1,\tilde Y^1, \tilde Z^1), \dots, (\tilde X^N,\tilde Y^N, \tilde Z^N)$ of $(X,Y,Z)$ such that for each $i$, $(\tilde X^i, \tilde Y^i, \tilde Z^i)$ solves Equation \eqref{eq:MkVFBSDE} with driving Brownian motion $W^i$ and initial condition $x^i_0$.
This can be done when the McKean-Vlasov FBSDE~\eqref{eq:MkVFBSDE} has a unique solution, and thus the associated law $\cL({X_u},{Y_u})$ is unique at each time $u \in [0,T]$.
By \cite[Theorem 4.24]{MR3752669}, the FBSDE \eqref{eq:MkVFBSDE} is uniquely solvable for $T$ small.
The following lemma is a central element of our argument.
Recall the notation $\tilde \X: = (\tilde X^1, \dots,\tilde X^N)$ and $\tilde \Y:=(\tilde Y^1, \dots, \tilde Y^N)$.
\begin{lemma}
\label{lem: fundamental lemma}
 	If \ref{b1}-\ref{b2} are satisfied,
	then there are positive constants $C$ and $c(L_f)$ depending only on $L_f$ such that if $T\le c(L_f)$, then for every $0\le t\le T$ it holds that
	\begin{align}
	\notag
		&E\Big[\mathcal{W}^2_2(\cL(X_t, Y_t), L^N(\X_t, \Y_t)) \Big]\\
		\label{eq:fundamental inequality for y}
		 &\qquad\qquad\leq C E\Big[\mathcal{W}^2_2(\cL(X_t, Y_t), L^N(\tilde \X_t, \tilde \Y_t)) + \mathcal{W}^2_2(\cL({X_T}),L^N(\tilde \X_T))\Big].
	\end{align}
\end{lemma}
\begin{proof}
	Applying It\^{o}'s formula to the process $e^{\beta t}|\tilde Y^{i}_t-Y^{i,N}_t|^2$ for some $\beta \ge 0$ to be determined later, we have
	\begin{align*}
		&e^{\beta t}|\tilde Y^{i}_t-  Y^{i,N}_t|^2\\
		 & = e^{\beta T}|G(\tilde X^{i}_T,\cL({X_T})) - G(X^{i,N}_T,L^N(\X_T))|^2-2\sum_{k=1}^N\int_t^Te^{\beta u}(\tilde Y^{i}_u-Y^{i,N}_u)(\delta_{k,i}\tilde Z^{i}_u-Z^{i,k,N}_u)dW^k_u\\
		&+2\int_t^Te^{\beta u}(\tilde Y^{i}_u-Y^{i,N}_u)\left[F_u(\tilde X^{i}_u,\tilde Y^{i}_u,\tilde Z^{i}_u,\cL({X_u},{Y_u}))-F_u(X^{i,N}_u,Y^{i,N}_u,Z^{i,i,N}_u, L^N(\X_u,\Y_u)) \right]\,du\\
		&  -\sum_{j=1}^N\int_t^Te^{\beta u}|Z^{i,j,N}_u-\delta_{ij}Z^i_u|^2\,du -\int_t^T\beta e^{\beta u}|\tilde Y^{i}_u-Y^{i,N}_u|^ 2du.
	\end{align*}
	By Lipschitz continuity of $F$ and $G$, then applying Young's inequality with 
	a strictly positive constant $a$ to be set below we get 
	\begin{align}
	\nonumber &e^{\beta t}|\tilde Y^{i}_t-Y^{i,N}_t|^2
	 \leq 2e^{\beta T}L_f|\tilde X^{i}_T-X^{i,N}_T|^2 +2e^{\beta T}L_f\mathcal{W}^2_2(\cL({X_T}),L^N(\X_T)) \\
	\nonumber &\quad-2\sum_{k=1}^N\int_t^Te^{\beta u}(\tilde Y^{i}_u-Y^{i,N}_u)(\delta_{k,i}\tilde Z^{i}_u-Z^{i,k,N}_u)dW^k_u +\int_t^Te^{\beta u}L_f|\tilde X^{i}_u-X^{i,N}_u|^2du\\\notag
	&\quad+\int_t^T e^{\beta u}\left(L_fa+4L_f-\beta\right)|\tilde Y^{i}_u-Y^{i,N}_u|^2du - \sum_{j=1}^N\int_t^T|Z^{i,j,N}_u-\delta_{ij}\tilde Z^i_u|^2\,du\\\notag
	&\quad + L_f\int_t^T e^{\beta u} \mathcal{W}_2^2(L^N(\X_u,\Y_u),\cL(X_u,{Y_u}))\,du + \frac{L_f}{a}  \int_t^Te^{\beta u}|\tilde Z^{i}_u-Z^{i,i,N}_u|^2du.
	\end{align}
	Letting $a >L_f$ and $\beta=L_fa+ 4L_f$,
	and taking conditional expectation on both sides above, we have the estimate
	\begin{align}
	\nonumber
		|\tilde Y^{i}_t-Y^{i,N}_t|^2 
		&\le	2e^{\beta T}L_fE\Big[|\tilde X^{i}_T-X^{i,N}_T|^2 +\mathcal{W}^2_2(L^N(\X_T),\cL({X_T}))\\
	\label{eq:first estimfory}
		&\quad + \int_t^T \left( |\tilde X^{i}_u-X^{i,N}_u|^2 + \mathcal{W}_2^2(L^N(\X_u,\Y_u),\cL(X_u,{Y_u})) \right) \,du \mid \mathcal{F}_t^N\Big].
	\end{align}
	On the other hand, for every $0\le s\le t\le T$, by Lipschitz continuity of $B$, the forward equation yields the estimate 
	\begin{align}
	\label{eq:pathwise for x}
		|\tilde X^i_t - X^{i,N}_t| \le L_f\int^t_{0} \left( |\tilde X^i_u - X^{i,N}_u| + |\tilde Y^i_u - Y^{i,N}_u| + \mathcal{W}_2(L^N(\X_u,\Y_u),\cL(X_u,{Y_u})) \right)\,du.
	\end{align}
	Adding up the squared power of the above with \eqref{eq:first estimfory} yields 
	\begin{align*}
		&|\tilde X^{i}_t-X^{i,N}_t|^2+|\tilde Y^{i}_t-Y^{i,N}_t|^2 
	 \le  C_{L_f,T}E\Big[\mathcal{W}^2_2(L^N(\X_T),\cL({X_T})) \\
	 &+ \int_0^T \left( |\tilde X^{i}_u-X^{i,N}_u|^2 + |\tilde Y^{i}_u-Y^{i,N}_u|^2 + \mathcal{W}_2^2(L^N(\X_u,\Y_u),\cL(X_u,{Y_u})) \right)\,du 
	 \mid \mathcal{F}_t^N\Big].
	\end{align*}
	If $T<1\wedge \frac{1}{C_{L_f,T}}$, we then have 
	\begin{align*}
		&E\big[|\tilde X^{i}_t-X^{i,N}_t|^2+|\tilde Y^{i}_t-Y^{i,N}_t|^2 \big]\\
		&\qquad \le  C_{L_f,T,1}E\Big[\mathcal{W}^2_2(L^N(\X_T),\cL({X_T})) + \int_0^T \mathcal{W}_2^2(L^N(\X_u,\Y_u),\cL(X_u,{Y_u}))\,du
	 \Big]
	\end{align*}
	for a constant $C_{L_f,T,1}$ which depends only on $L_f,T$.
	Coming back to the forward equation, it follows by the definition of the $2$-Wasserstein distance, by \eqref{eq:pathwise for x} and by Gronwall's inequality that 
	\begin{align*}
		\mathcal{W}_2^2(L^N(\tilde \X_T), L^N(\X_T)) &\le \frac 1N \sum_{i=1}^N |\tilde X^i_T - X^{i,N}_T|^2\\
		& \le e^{2L_fT}\int_0^T\Big(\frac 1N \sum_{i=1}^N|\tilde Y_u^i - Y^{i,N}_u|^2 +  \mathcal{W}_2^2(L^N(\X_u,\Y_u),\cL(X_u,{Y_u})) \Big) \,du.
	\end{align*}
	Therefore, we can continue the estimation of $|\tilde X^{i}_t-X^{i,N}_t|^2+|\tilde Y^{i}_t-Y^{i,N}_t|^2$ by
	\begin{align}
	\notag
		&E\big[ |\tilde X^{i}_t-X^{i,N}_t|^2+|\tilde Y^{i}_t-Y^{i,N}_t|^2 \big]\\\notag
		&\le  C_{L_f,T,1}E\bigg[\mathcal{W}^2_2(L^N(\tilde \X_T),\cL({X_T})) + \mathcal{W}^2_2(L^N(\tilde \X_T),L^N(\X_T))
		\\\notag
		&\qquad \qquad + \int^T_{0} \mathcal{W}_2^2(L^N(\X_u,\Y_u),\cL(X_u,{Y_u}))\,du 
	 	\bigg] 
		\\\notag
		 &\leq C_{L_f,T,1}\vee 2e^{2L_fT}E\bigg[\mathcal{W}^2_2(\cL({X_T}), L^N(\tilde \X_T))
		  + \int^T_{0}\Big(\frac 1N\sum_{i=1}^N\big\{|\tilde Y_u^i - Y^{i,N}_u|^2\\
		  &\label{eq:useful Y estimate} 
		  \qquad\qquad + |\tilde X_u^i - X^{i,N}_u|^2\big\} +  \mathcal{W}_2^2(L^N(\X_u,\Y_u),\cL(X_u,{Y_u})) \Big)\,du \bigg].
	\end{align}
	Thus, further assuming $T\le \frac{1}{C_{L_f,T,1}\vee e^{L_fT}}$ yields
	\begin{align*}
	&E\left[\mathcal{W}^2_2(L^N(\X_t, \Y_t), L^N(\tilde \X_t, \tilde \Y_t)) \right]
		\le E\Big[\frac1N\sum_{i=1}^N	(|\tilde X^{i}_t-X^{i,N}_t|^2+|\tilde Y^{i}_t-Y^{i,N}_t|^2)\Big] \\
		  & \qquad\qquad \leq C_{L_f,T,2}E\Big[\mathcal{W}^2_2(\cL({X_T}), L^N(\tilde \X_T)) + \int^T_{0} \mathcal{W}_2^2(L^N(\X_u,\Y_u),\cL(X_u,{Y_u}))\,du\Big].
	\end{align*}
	By the triangle inequality we can therefore deduce that
	\begin{align*}
		&E\Big[\mathcal{W}^2_2( L^N(\X_t, \Y_t),\cL(X_t, Y_t)) \Big] \\
		 &\leq E\Big[\mathcal{W}^2_2(L^N(\tilde \X_t, \tilde \Y_t), \cL(X_t, Y_t))\Big] + 	E\Big[ \mathcal{W}^2_2(L^N(\tilde \X_t, \tilde \Y_t),L^N(\X_t, \Y_t))  \Big]
		\\
		& \le E\Big[ \mathcal{W}^2_2(L^N(\tilde \X_t, \tilde \Y_t), \cL(X_t, Y_t)) \Big]
		\\
		& \quad + C_{L_f,T,2}E\Big[ \mathcal{W}^2_2(\cL({X_T}), L^N(\tilde \X_T)) + \int^T_{0} \mathcal{W}_2^2(L^N(\X_u,\Y_u), \cL(X_u,{Y_u}))\,du \Big]
	\end{align*}
	from which we derive \eqref{eq:fundamental inequality for y}, assuming $T<1/C_{L_f,T,2}$.
\end{proof}

\begin{proof}(of Theorem \ref{thm:mom bound syst})
	The bound \eqref{eq:mom bound y system} follows by Lemmas \ref{lem: fundamental lemma} and \ref{lem: moment d+5 system}.
	In fact, from Lemma \ref{lem: fundamental lemma} if $T$ is small enough that $T<1/C_{L_f,T,2}$, then for every $t \in [0,T]$ it holds that 
	\begin{align*}
		&E\Big[\cW^2_2( L^N(\X_t, \Y_t),\cL(X_t, Y_t)) \Big]\\
		 &\qquad \le CE\Big[ \cW^2_2( L^N(\tilde \X_t, \tilde \Y_t), \cL(X_t, Y_t)) \Big] 
		+ CE\Big[\cW^2_2(L^N(\tilde \X_T), \cL(X_T)) \Big]
		\\
		&\qquad \le C(r_{N,m+\xdim,k}  + r_{N,\xdim,k} ) 
	\end{align*}
	where the second inequality follows by \cite[Theorem 1]{Fou-Gui15} which can be applied thanks to Lemma \ref{lem: moment d+5 system}.
	To prove \eqref{eq:process conv}, first observe that by assumption \ref{b1} and Gronwall's inequality we readily have
	\begin{equation}
	\label{eq: sde conv syst}
		|X^{1,N}_t - X^1_t| \le e^{L_fT}\int_0^t \left( |Y^{1,N}_u -Y^1_u | + \cW_2(L^N(\X_u, \Y_u), \cL(X_u, Y_u)) \right) \,du
	\end{equation}
	for all $0\le t\le T$.
	On the other hand, by It\^{o}'s formula applied to the process $| Y^{1,N}_t-Y^1_t|^2$ as in the proof of Lemma \ref{lem: fundamental lemma}, and then the inequality $2xy \le \varepsilon x^2 + y^2/\varepsilon$ with the constant $\varepsilon := 1/2$, we have
	\begin{align}
	\notag
		&| Y^{1,N}_t-Y_t^1|^2+\sum_{j}^N\int_t^T|Z^{1,j,N}_s-\delta_{1j}Z^1_s|^2\,ds\\\notag
		&\le L_f\left(|X^{1,N}_T-X^{1}_T|^2 + \cW_2^2(L^N(\X_T),\cL(X_T)) \right) -2\sum_{k=1}^N\int_t^T( Y^{1,N}_s-Y^1_s)(Z^{1,k,N}_s-\delta_{k,1} Z^1_s)dW^k_s
		\\
		\notag
		&\quad +\int_t^T \left( \frac{1}{2} | Z^{1,N}_s-Z^1_s|^2 + |X^{1,N}_s-X^{1}_s|^2 \right) ds
		+\int_t^T (3L_f^2 + L_f ) | Y^{1,N}_s-Y^{1}_s|^2ds
		\\
		&\quad+\int_t^TL_f\cW^2_2(L^N(\X_s,\Y_s),\cL(X_s,{Y_s}))\,ds.
		\label{eq:estim y&z}
	\end{align}
	Thus, it follows by Gronwall's inequality that
	\begin{align}
	\notag
		&|Y^{1,N}_t - Y^1_t|^2 +E\left[\int_t^T|Z^{1,1,N}_u - \tilde Z^1_u\,du \mid \cF^N_t \right]\\\notag
		 &\le C_{L_f,T}E\bigg[ \cW_2^2(L^N(\X_T), \cL(X_T)) + \sup_{u \in [s,T]}|X^{1,N}_u-X^{1}_u|^2\\\notag
		&\quad + \int_t^T| Y^{1,N}_u-Y^{1}_u|^2du+\int_t^T\cW^2_2(L^N(\X_u,\Y_u),\cL(X_u,{Y_u}))\,du \mid \cF_t^N \bigg] \\\notag
		& \le C_{L_f,T} E\bigg[ \cW^2_2(L^N(\X_T), \cL( X_T)) \\
		\label{eq: estim DY before trianglular}
		&\quad + \int^T_{0}\cW^2_2(L^N(\X_u, \Y_u), \cL(X_u, Y_u))\,du + \int^T_0|Y^{1,N}_u - Y_u^1|^2\,du  \mid \cF_t^N \bigg],
	\end{align}
	where the second inequality follows by  \eqref{eq: sde conv syst} and $C_{L_f,T}>0$ is a constant depending only on $L_f$ and $T$.
	If $T$ is small enough, then we have
	\begin{align*}
		&\sup_{t \in [s,T]}E[|Y^{1,N}_t - Y_t^1 |^2] \\
		&\le C_{L_f,T} E\bigg[ \cW^2_2(L^N(\X_T), \cL( X_T)) + \int_0^T\cW^2_2(L^N(\X_u, \Y_u), \cL(X_u, Y_u))\,du \bigg]\\
				&\le C(r_{N,q+\xdim, k} + r_{N,\xdim,k})
	\end{align*}
	where the last inequality follows from \eqref{eq:mom bound y system}, and
	where we also used that 
	\begin{equation}
	\label{eq:Wars pro vs non pro}
		\cW^2_2(L^N(\X_T), \cL( X_T))  \le  \cW^2_2(L^N(\X_T, \Y_T), \cL(X_T, Y_T)).
	\end{equation}
	Thus, using \eqref{eq: sde conv syst} leads to
	\begin{align}
	\nonumber
		E\Big[\sup_{t \in [s,T]}|X^{1,N}_t - X^1_t|^2 \Big]
		\label{eq:estimation x in proof}
		&\le C\big( r_{N,q+\xdim, k} + r_{N,\xdim,k}\big).
	\end{align}
	Finally, coming back \eqref{eq: estim DY before trianglular} yields the bound for $\|Z^{1,1,N} - Z^1\|_{\mathcal{H}^2(\mathbb{R}^\xdim\times \mathbb{R}^d)}$.
	This concludes the proof.
\end{proof}

\subsection{Propagation of chaos under monotonicity conditions}
The next result shows that under additional monotonicity conditions Theorem \ref{thm:mom bound syst} can be extended to arbitrary time duration $T>0$.
These monotonicity conditions are classical in the analysis of FBSDE, they are for instance used in \cite{Peng-Wu99,Delarue,Ben-Yam-Zhang15}.
Here, it is important to distinguish the Lipschitz--constant of $B$ in each of its arguments.
Thus, in \ref{b1}, we write
\begin{equation*}
	|B_t(x,y,\xi)-B_t(x',y',\xi')|\le L_{B,x}|x-x'|+L_{B,y}|y-y'| +  L_{B,\xi}\mathcal{W}_2(\xi, \xi') 
\end{equation*}
for some $L_{B,x}, L_{B,y}, L_{B,\xi}>0$ and all $x,x' \in \mathbb{R}^\xdim$, $y,y'\in \mathbb{R}^\ydim$ and $\xi,\xi'\in \mathcal{P}_2(\mathcal{R}^\xdim\times\mathbb{R}^\ydim)$.
\begin{theorem}[Monotonicity conditions]
\label{thm:chaos.T.arbitray}
	Assume that the conditions \ref{b1}-\ref{b2}, \ref{a5} are satisfied and that there is $k>2$ such that $E[|x^1_0|^k]<\infty$.
	Further assume that the McKean-Vlasov FBSDE \eqref{eq:MkVFBSDE} admits a unique solution $(X,  Y, Z)\in \mathcal{S}^2(\mathbb{R}^\xdim)\times \mathcal{S}^2(\mathbb{R}^\ydim)\times \mathcal{H}^2(\mathbb{R}^{\ydim\times d}) $ and:
	\begin{enumerate}[label = (\textsc{B3}), leftmargin = 30pt]
		\item there is a constant $K_B>0$ such that the following monotonicity property holds
		\begin{align*}
	 		(x - x')\cdot \Big( B_t(x,y,\xi)-B_t(x',y,\xi)\Big) \le -K_B|x-x^\prime|^2
		\end{align*}
	 	for all $x,x^\prime \in \mathbb{R}^\xdim$ and $(t,y,\xi) \in [0,T]\times\mathbb{R}^\ydim\times \cP_2(\mathbb{R}^\xdim\times \mathbb{R}^\xdim)$.
		\label{b3}
	\end{enumerate} 
	If the constant $K_B$ satisfies \eqref{eq:choiceK_b.k} and
	\begin{equation*}
		K_B > 8T(L_G^2 + L_FT) (L_{B,\xi} + L_{B,y})^2\exp\Big\{ 2L_F\Big(6 + L_F\Big)\Big\}  + 2L_{B,\xi} ,
	\end{equation*}	
	then it holds
	\begin{equation*}
	 	\sup_{t\in [0,T]}E\bigg[\cW_2^2\big(L^N(\underline{X}_t, \underline{Y}_t), \cL(X_t, Y_t) \big)\bigg] \le Cr_{N,q+\xdim, k} 
	\end{equation*}
	and 
	\begin{equation*}
	 	\sup_{t\in [0,T]}\Big(E\Big[ |X^{i,N}_t - X^i_t|^2\Big] + E\Big[|Y^{i,N}_t - Y^i_t|^2 \Big]\Big) + E\bigg[\int_0^T|Z^{i,i,N}_t -  Z^i_t|^2\,dt \bigg] \le Cr_{N,q+\xdim, k}
	\end{equation*}
	for all $t\in [0,T]$,  $N\in \mathbb{N}$ and for a constant $C>0$.
\end{theorem}
\begin{proof}
	As in the proof of Theorem \ref{thm:mom bound syst}, let $(\tilde X^i,\tilde Y^i, \tilde Z^i)_{1\le i\le N}$ be $N$ i.i.d. copies of the solution $(X, Y, Z)$ of the Mckean--Vlasov equation \eqref{eq:MkVFBSDE}.
	We will use the shorthand notation $\Delta X^i_t:= X^{i,N}_t - \tilde X^i_t$, $\Delta Y^i_t: = Y^{i,N}_t - \tilde Y^i_t$ and $\Delta Z^{i,j}_t:= Z^{i,j,N}_t - \delta_{\{i=j\}}\tilde Z^{i}_t$.
	Applying It\^o's formula, we have
	\begin{align*}
		|\Delta X^{i}_t|^2 &= 2\int_0^T\Delta X^i_u\cdot \Big( B_u(X^{i,N}_u, Y^{i,N}_u, L^N(\underline{X}_u, \underline{Y}_u)) - B_u(\tilde X^i_u, \tilde Y^i_u, \cL(X_u, Y_u)\Big)\,du\\
			& = 2\int_0^t\Delta X^i_u\cdot \Big( B_u(X^{i,N}_u, Y^{i,N}_u, L^N(\underline{X}_u, \underline{Y}_u)) - B_u(\tilde X^i_u, Y^{i,N}_u, L^N(\underline{X}_u, \underline{Y}_u))\Big)\,du\\
			&\quad + 2\int_0^t\Delta X^i_u\cdot \Big(B_u(\tilde X^i_u, Y^{i,N}_u, L^N(\underline{X}_u, \underline{Y}_u)) - B_u(\tilde X^i_u, \tilde Y^i_u, \cL(X_u, Y_u))\Big)\,du\\
			&\le 2\int_0^t-K_B|\Delta X_u^i|^2 + {\color{black}L_{B,y}}|\Delta X^i_u||\Delta Y^i_u| + {\color{black}L_{B,\xi}}|\Delta X^i_u|\Big(\frac1N\sum_{j=1}^N|\Delta X^j_u|^2 + |\Delta Y^j_u|^2 \Big)^{1/2}\\
			&\quad  + {\color{black}L_{B,\xi}}|\Delta X^i_u|\cW_2(L^N(\underline{\tilde X}_u,\underline{\tilde Y}_u), \cL(X_u, Y_u) )\,du
	\end{align*}
	where the latter inequality follows by the monotonicity property and Lipschitz--continuity of $B$ and triangular inequality applied on the Wasserstein distance.
	Now, applying Young's inequality with some $\varepsilon>0$, we obtain
	\begin{align}
	\notag
		|\Delta X^i_t|^2
		&\le 2\int_0^t{\color{black}\Big(\frac{L_{B,y} + L_{B,\xi}}{2\varepsilon} + L_{B,\xi} -K_B \Big)}|\Delta X^i_u|^2 + \frac{L_{B,\xi}}{2}\frac1N\sum_{j=1}^N|\Delta X^j_u|^2\,du\\
		\label{eq:monon.boundx.2}
		&\quad  + \int_0^t\varepsilon L_{B,y}|\Delta Y^i_u|^2 + \varepsilon L_{B,\xi}\frac1N\sum_{j=1}^N|\Delta Y^j_u|^2 
		 + L_{B,\xi}\cW^2_2(L^N(\underline{\tilde X}_u, \underline{\tilde Y}_u), \cL(X_u, Y_u))\,du.
	\end{align}
	Thus, taking the average on both sides gives
	\begin{align*}
		\frac1N\sum_{j=1}^N|\Delta X^j_t|^2 &\le 2\int_0^t{\color{black}\Big(\frac{L_{B,y} + L_{B,\xi}}{2\varepsilon} + 2L_{B,\xi} -K_B \Big)}\frac1N\sum_{j=1}^N|\Delta X^j_u|^2\,du\\
		&\quad +  \int_0^t\varepsilon(L_{B,\xi} + L_{B,y}) \frac1N\sum_{j=1}^N|\Delta Y^j_u|^2  + L_{B,\xi}\cW^2_2(L^N(\underline{\tilde X}_u, \underline{\tilde Y}_u), \cL(X_u, Y_u))\,du.
	\end{align*}
	Next, we apply Gronwall's inequality to arrive at the bound
	\begin{align}
	\notag
		&\frac1N\sum_{j=1}^N|\Delta X^j_t|^2\\\label{eq:mono.boundx1}
		 &\le e^{2\delta(\varepsilon)T}\int_0^t \varepsilon(L_{B,\xi} + L_{B,y}) \frac1N\sum_{j=1}^N|\Delta Y^j_u|^2 
		 + L_{B,\xi}\cW^2_2(L^N(\underline{\tilde X}_u, \underline{\tilde Y}_u), \cL(X_u, Y_u))\,du
	\end{align}
	where we introduced the constant
	\begin{equation*}
		\delta(\varepsilon) := \frac{L_{B,y} + L_{B,\xi}}{2\varepsilon} + 2L_{B,\xi} -K_B .
	\end{equation*}
	Let us now turn to the backward process.
	Here again, we apply It\^o's formula to get
	\begin{align*}
		|\Delta Y^i_t|^2& = |G(X^{i,N}_T, L^N(\underline{X}_T)) - G(\tilde X_T, \cL(X_T))|^2\\
		&+2\int_t^T\Delta Y^i_u\cdot \Big(F_u(X^{i,N}_u, Y^{i,N}_u, Z^{i,i,N}_u, L^N(\underline{X}_u, \underline{Y}_u)- F_u(\tilde X^i_u, \tilde Y^i_u, \tilde Z^i_u, \cL(X_u, Y_u))) \Big)\,du\\
		&-\sum_{j=1}^N\int_t^T|\Delta Z^{i,j}_u |^2\,du - \sum_{j=1}^N\int_t^T2\Delta Y^i_u\Delta Z^{i,j}_u \,dW^j_u\\
		&\le 2L_G^2\Big(|\Delta X^i_T|^2 + \frac1N\sum_{j=1}^N|\Delta X^j_T|^2 + \cW_2^2(L^N(\underline{\tilde{X}}_T),\cL(X_T))\Big) + 2L_F\int_t^T|\Delta Y^i_u|\Big\{ |\Delta X^i_u| + |\Delta Y^i_u| + |\Delta Z^{i,i}_u|\\
		&\quad + \Big(\frac1N\sum_{j=1}^N|\Delta X^j_u|^2 + |\Delta Y^j_u|^2\Big)^{1/2} + \cW_2(L^N(\underline{\tilde{X}}_u, \underline{\tilde{Y}}_u), \cL(X_u, Y_u))\Big\}\,du\\
		&\quad-\sum_{j=1}^N\int_t^T|\Delta Z^{i,j}_u |^2\,du - \sum_{j=1}^N\int_t^T2\Delta Y^i_u\Delta Z^{i,j}_u \,dW^j_u
	\end{align*}
	where we used Lipschitz--continuity of $F$ and $G$.
	Now, we apply Young's inequality with some constant $\eta>0$ and then take conditional expectation on both sides (the martingale property follows from integrability properties proved above) to arrive at
	\begin{align}
	\notag
		|\Delta Y^i_t|^2& + (1-\eta L_F)E\Big[\sum_{j=1}^N\int_t^T|\Delta Z^{i,j}_u|^2\,du\mid \cF^N_t\Big]
		 \le 2L_G^2E\Big[|\Delta X^i_T|^2 + \frac1N\sum_{j=1}^N|\Delta X^j_T|^2\mid \cF^N_t\Big]\\\notag
			&\quad+ 2L_FE\bigg[\int_t^T\Big(5 + \frac1\eta\Big)|\Delta Y^i_u|^2 +\frac1N\sum_{j=1}^N|\Delta Y^j_u|^2 + |\Delta X^i_u|^2 +\frac1N\sum_{j=1}^N|\Delta X^j_u|^2\\\label{eq:monotone.bou}
			&\qquad + \cW^2_2(L^N(\underline{\tilde{X}}_u, \underline{\tilde{Y}}_u),\cL(X_u,Y_u))\,du\mid \cF^N_t\bigg] + 2L_G^2E\Big[\cW_2^2(L^N(\underline{\tilde{X}}_T),\cL(X_T)) \mid \cF^N_t\Big].
	\end{align}
	Averaging on both sides and choosing $\eta$ small enough that $1 - \eta F>0$ yields
	\begin{align*}
		\frac1N\sum_{j=1}^N|\Delta Y^j_t|^2& + (1-\eta L_F)\frac1N\sum_{j=1}^NE\Big[\int_t^T|\Delta Z^{j,j}_u|^2\,du\mid \cF^N_t\Big] \le 4L_G^2E\Big[\frac1N\sum_{j=1}^N|\Delta X^j_T|^2\mid \cF^N_t\Big]\\
			&+ 2L_FE\bigg[\int_t^T\Big(6 + \frac1\eta\Big)\frac1N\sum_{j=1}^N|\Delta Y^j_u|^2 + 2\frac1N\sum_{j=1}^N|\Delta X^j_u|^2\\
			&\quad + \cW^2_2(L^N(\underline{\tilde X}_u, \underline{\tilde Y}_u),\cL(X_u,Y_u))\,du\mid \cF^N_t\bigg] + 2L_G^2E\Big[\cW_2^2(L^N(\underline{\tilde{X}}_T),\cL(X_T)) \mid \cF^N_t\Big].
	\end{align*}
	We will subsequently apply Gronwall's inequality, take expectation on both sides and then integrate in time.
	Thus, due to Fubini's theorem we have
	\begin{align*}
		&E\bigg[\frac1N\sum_{j=1}^N\int_0^T|\Delta Y^j_t|^2\,dt\bigg]  \le 4L_G^2Te^{\bar \delta (\eta)T}E\bigg[\frac1N\sum_{j=1}^N|\Delta X^j_T|^2\bigg] + 2TL_G^2e^{\bar \delta (\eta)T}E\Big[\cW_2^2(L^N(\underline{\tilde{X}}_T),\cL(X_T)) \Big]\\
		& \qquad	+ 2L_FTe^{\bar\delta (\eta)T}E\bigg[\int_0^T 2\frac1N\sum_{j=1}^N|\Delta X^j_u|^2 + \cW^2_2(L^N(\underline{\tilde X}_u, \underline{\tilde Y}_u),\cL(X_u,Y_u))\,du\bigg]
	\end{align*}
	where we introduced the constant
	\begin{equation*}
		\bar \delta (\eta) := 2L_F\Big(6 + \frac1\eta\Big).
	\end{equation*}
	Using \eqref{eq:mono.boundx1}, we further bound the above as
	\begin{align*}
		&E\bigg[\frac1N\sum_{j=1}^N\int_0^T|\Delta Y^j_t|^2\,dt\bigg]  \le \Gamma_{\varepsilon,T,G,B,F}E\bigg[\frac1N\sum_{j=1}^N\int_0^T|\Delta Y^j_t|^2\,dt\bigg] + 2TL_G^2E\Big[\cW_2^2(L^N(\underline{\tilde{X}}_T),\cL(X_T)) \Big]\\
		& \qquad	+ 4T(L_G^2L_{B,\xi} +L_FT)(L_{B,\xi}+1)e^{\bar\delta (\eta)T}e^{2\delta(\varepsilon)T}E\bigg[\int_0^T\cW^2_2(L^N(\underline{\tilde X}_u, \underline{\tilde Y}_u),\cL(X_u,Y_u))\,du\bigg]
	\end{align*}
	with 
	\begin{equation*}
		\Gamma_{\varepsilon,T,G,B,F} := 4\varepsilon T(L_G^2 + L_FT)  e^{\bar\delta (\eta)T}e^{2\delta(\varepsilon)T}(L_{B,\xi} + L_{B,y}).
	\end{equation*}
	First choose $\varepsilon$ small enough that
		\begin{equation*}
			4\varepsilon T(L_G^2 + L_FT)  e^{\bar\delta (\eta)T}(L_{B,\xi} + L_{B,y}) <1.
	\end{equation*}
	This $\varepsilon$ does not depend on $K_B$.
	With such an $\varepsilon$ at hand, choose $K_B$ large enough that $\delta(\varepsilon)\le 0$.
	Thus, we need 
	\begin{equation*}
		K_B \ge T(L_G^2 + L_FT)  e^{\bar\delta (\eta)T}(L_{B,\xi} + L_{B,y})^2  + L_{B,\xi}.
	\end{equation*}
	This implies that $\Gamma_{\varepsilon, T, G, B, F}<1$.
	Hence, we have
	\begin{align*}
		E\bigg[\frac1N\sum_{j=1}^N\int_0^T|\Delta Y^j_t|^2\,dt\bigg]  &\le \frac{4T(L_G^2L_{B,\xi} +2L_FT)(L_{B,\xi}+1)e^{\bar\delta (\eta)T}}{1 - \Gamma_{\varepsilon,T,G,B,F}}E\bigg[\int_0^T\cW^2_2(L^N(\underline{\tilde X}_u, \underline{\tilde Y}_u),\cL(X_u,Y_u))\,du\bigg]\\
		&\quad + \frac{2TL_G^2}{1-\Gamma_{\varepsilon,T,G,B,F}}E\Big[\cW_2^2(L^N(\underline{\tilde{X}}_T),\cL(X_T)) \Big].
	\end{align*}
	This also implies, due to \eqref{eq:mono.boundx1}, that
	\begin{align*}
		E\bigg[\frac1N\sum_{j=1}^N\int_0^T|\Delta X^j_t|^2\,dt\bigg] & \le CE\bigg[\int_0^T\cW^2_2(L^N(\underline{\tilde X}_u, \underline{\tilde Y}_u),\cL(X_u,Y_u)))\,du\bigg]\\
		&\quad + CE\Big[\cW_2^2(L^N(\underline{\tilde{X}}_T),\cL(X_T)) \Big]
	\end{align*}
	for some constant $C>0$.

	We will now use these inequalities to show the claimed convergence results.
	Going back to \eqref{eq:monon.boundx.2} and (recalling the choice of $\varepsilon$), we have
	\begin{align}
	\notag
		E[|\Delta X^i_t|^2] &\le E\bigg[  2\int_0^t \frac{L_{B,\xi}}{2}\frac1N\sum_{j=1}^N|\Delta X^j_u|^2\,du\\\notag
		&\quad  + \int_0^t\varepsilon L_{B,y}|\Delta Y^i_u|^2 + \varepsilon L_{B,\xi}\frac1N\sum_{j=1}^N|\Delta Y^j_u|^2 
		 + L_{B,\xi}\cW^2_2(L^N(\underline{\tilde X}_u, \underline{\tilde Y}_u), \cL(X_u, Y_u))\,du\bigg]\\\label{eq:monon.boundx3}
		&\le CE\bigg[\int_0^T\cW^2_2(L^N(\underline{\tilde X}_u, \underline{\tilde Y}_u), \cL(X_u, Y_u))\,du\bigg] + CE\Big[\cW_2^2(L^N(\underline{\tilde{X}}_T),\cL(X_T)) \Big]
		+\varepsilon L_{B,\xi}E\bigg[\int_0^t|\Delta Y^i_u|^2\,du\bigg].
	\end{align}
	Plugging this bound in \eqref{eq:monotone.bou}, gives
	\begin{align}
	\notag
		E[|\Delta Y^i_t|^2] &  \le 2L_G^2e^{\bar \delta (\eta)}E\Big[|\Delta X^i_T|^2  + \frac1N\sum_{j=1}^N|\Delta X^j_T|^2\Big] + 2L_G^2e^{\bar\delta(\eta)}E\Big[\cW_2^2(L^N(\underline{\tilde{X}}_T),\cL(X_T)) \Big]\\\notag
		&\quad + 2L_Fe^{\bar \delta(\eta)}E\bigg[\int_t^T\frac1N\sum_{j=1}^N|\Delta Y^j_u|^2  + |\Delta X^i_u|^2 +\frac1N\sum_{j=1}^N|\Delta X^j_u|^2 + \cW^2_2(L^N(\underline{\tilde X}_u, \underline{\tilde Y}_u),\cL(X_u,Y_u))\,du\bigg]\\\notag
		&\le CE\bigg[\int_0^T\cW^2_2(L^N(\underline{\tilde X}_u, \underline{\tilde Y}_u), \cL(X_u, Y_u))\,du\bigg] + CE\Big[\cW_2^2(L^N(\underline{\tilde{X}}_T),\cL(X_T)) \Big]\\
		&\quad	+2\varepsilon L_{B,\xi}Te^{\bar\delta(\eta)T}(L_F+L_G)E\bigg[\int_0^T|\Delta Y^i_u|^2\,du\bigg].
	\end{align}
	We now integrate in time on both sides, we use Fubini's theorem and further choose $\varepsilon$ small enough that $2\varepsilon L_{B,\xi}Te^{\bar\delta(\eta)T}(L_F+L_G)<1$.
	This allows to obtain the bound
	\begin{align*}
		E\bigg[\int_0^T|\Delta Y^i_u|^2\,du\bigg]\le CE\bigg[\int_0^T\cW^2_2(L^N(\underline{\tilde X}_u, \underline{\tilde Y}_u), \cL(X_u, Y_u)))\,du\bigg] + CE\Big[\cW_2^2(L^N(\underline{\tilde{X}}_T),\cL(X_T)) \Big].
	\end{align*}
	Thus, due to \eqref{eq:monon.boundx3}, we have
	\begin{align*}
		E\Big[|\Delta X^i_t|^2\Big]\le CE\bigg[\int_0^T\cW^2_2(L^N(\underline{\tilde X}_u, \underline{\tilde Y}_u), \cL(X_u, Y_u)))\,du\bigg] + CE\Big[\cW_2^2(L^N(\underline{\tilde{X}}_T),\cL(X_T)) \Big]
	\end{align*}
	for all $t\in [0,T]$.
	Going back once again to  \eqref{eq:monotone.bou} (after taking expectation and using Gronwall's inequality) allows to obtain the bound
	\begin{align*}
		E\bigg[|\Delta Y^i_t|^2 + \int_0^T|\Delta Z^i_u|\,du\bigg] &\le CE\bigg[\int_0^T\cW^2_2(L^N(\underline{\tilde X}_u, \underline{\tilde Y}_u), \cL(X_u, Y_u)))\,du\bigg]\\
		&\quad + CE\Big[\cW_2^2(L^N(\underline{\tilde{X}}_T),\cL(X_T)) \Big].
	\end{align*}
	Finally observe that by triangular inequality we have
	\begin{align*}
		&E\bigg[\cW^2_2(L^N(\underline{ X}_t, \underline{ Y}_t), \cL(X_t, Y_t)))\bigg]\\
		& \le 2\frac1N\sum_{j=1}^NE\bigg[|\Delta X^j_t|^2 + |\Delta Y^j_t|^2\bigg]+ 2E\bigg[\cW^2_2(L^N(\underline{\tilde X}_t, \underline{\tilde Y}_t), \cL(X_t, Y_t)))\bigg].
	\end{align*}
	This concludes the proof since the bound
	\begin{equation*}
		E\bigg[\cW^2_2(L^N(\underline{\tilde X}_t, \underline{\tilde Y}_t), \cL(X_t, Y_t))\bigg] + E\Big[\cW_2^2(L^N(\underline{\tilde{X}}_T),\cL(X_T)) \Big] \le Cr_{N,q+ \ell,k}.
	\end{equation*}
	follows by \cite[Theorem 1]{Fou-Gui15} and Lemma \ref{lem: moment d+5 system}.
\end{proof}

\subsection{Concentration estimates}
We conclude this section with some deviation and dimension-free concentration estimates to strengthen the above convergence results.
\label{sec:concen inequ}
\begin{theorem}
\label{thm:concentration syst}
	Assume that the conditions \ref{b1}-\ref{b2} and \ref{a5} are satisfied and that the McKean-Vlasov FBSDE \eqref{eq:MkVFBSDE} admits a unique solution $(X,Y,Z) \in \mathcal{S}^{2}(\mathbb{R}^\xdim)\times \mathcal{S}^{2}(\mathbb{R}^q)\times \mathcal{H}^{2}(\mathbb{R}^{q\times d})$. Then we have the following concentration estimations:
	\begin{enumerate}
		\item
		If there is $k>4$ such that $E[|x^1_0|^k]<\infty$, then for every $\varepsilon\in (0,\infty)$, $N\ge1$ it holds that
	\begin{align}
	\label{eq:conent bound}
	&	\sup_{t \in [0,T]}	P\left(\mathcal{W}_2^2(L^N(\X_t,\Y_t), \cL(X_t,Y_t)) \ge \varepsilon\right) \\\notag
	&	\le 
			C\big(a_{N,\frac{\varepsilon}{2}} 1_{\{\varepsilon<2\}} + b_{N,k,\frac{\varepsilon}{2}}+\frac{2}{\varepsilon}(r_{N,q+\xdim, k} + r_{N,\xdim,k}) \big)
	\end{align}
	for some constant $C>0$ which does not depend on $N,\varepsilon$,
	with $b_{N,k,\varepsilon} := N(N\varepsilon)^{-(k-\varepsilon)/2}$ and
	\begin{equation*}
		a_{N,\varepsilon} := \begin{cases}
			\exp(-cN\varepsilon^2) &\quad \text{if } q+\xdim<4\\
			\exp(-c N(\varepsilon/\log(2+1/\varepsilon))^2)& \quad \text{if } q+\xdim=4\\
			\exp(-cN\varepsilon^{(q+\xdim)/2})&\quad \text{if } q+\xdim>4
		\end{cases}
	\end{equation*}
	for two positive constants $C$ and $c$ depending only on $L_f$, $T$, $\sigma$, $k$ and $E[|x^1_0|^k]$.

	\item	There is a constant $c(L_f)>0$ such that if $T<c(L_f)$, then denoting by $\mu^N$ the $N$-fold product of the law $\mathcal{L}(X,Y)$ of $(X,Y)$, it holds that
		\begin{equation}
		\label{eq:gozlan}
			\mu^N\Big(H - \int H\,d\mu^N\ge\varepsilon \Big)\le e^{-K\varepsilon^2}
		\end{equation}
		for every $1$--Lipschitz continuous function $H \in C([0,T],\mathbb{R}^{\xdim+q})^N$ for some constant $K$ depending on $L_f, T$ and $\sigma$, but not on $(N, \xdim,q,d)$.
		If \ref{b2} is replaced by \ref{b2prime}, then \eqref{eq:gozlan} holds for all $T>0$.
\end{enumerate}
\end{theorem}

Let us start by the following lemma which gives a Talagrand $T_2$ inequality for the law of the solution of a foward-backward SDE.
Note that this result is not covered by \cite{T2bsde} since here, the system is fully coupled.
\begin{lemma}
\label{lem:Talagrand for system}
	Let $m_1, m_2\in \mathbb{N}$ and let $f:[0,T]\times \mathbb{R}^{m_1}\times \mathbb{R}^{m_2}\times \mathbb{R}^{m_2\times d} \to \mathbb{R}^{m_2}$, $b:[0,T]\times \mathbb{R}^{m_1}\times \mathbb{R}^{m_2}\to \mathbb{R}^{m_1}$ and $g:\mathbb{R}^{m_1}\to \mathbb{R}^{m_2}$ be such that $f(t, \cdot, \cdot,\cdot), b(t,\cdot,\cdot)$ and $g$ are three $L_f$--Lipschitz continuous function uniformly in $t$,
	and $\sigma \in \mathbb{R}^{m_1\times d}$ is a matrix satisfying \ref{a5}.
	Then there is a constant $c(L_f)>0$ depending only on $L_f$ such that if $T\le c(L_f)$, then
	the FBSDE
	\begin{equation}
	\label{eq:fbsde normal}
	\begin{cases}
		\displaystyle 
		X_t = x + \int_0^tb_u(X_u, Y_u)\,du + \sigma W_t\\
		\displaystyle 
		Y_t = g(X_T) + \int_t^Tf_u(X_u, Y_u,Z_u)\,du - \int_t^TZ_u\,dW_u
	\end{cases}
	\end{equation}
	admits a unique square integrable solution $(X,Y,Z)$, such that  $X$ and $Y$ have almost surely continuous paths 
	and 
	\begin{equation}
	\label{eq:talagrand for xy}
		\text{the law } \mathcal{L}(X, Y) \text{ of } (X, Y) \text{ satisfies } T_2(C_{x,y})
	\end{equation}
	for some constant $C_{x,y}$ (explicitly given in the proof) depending only on $L_f$, $T$ and $\sigma$, but which does not depend on $m_1, m_2$ and $d$.
	That is,
	\begin{equation*}
		\mathcal{W}_2(\mathcal{L}(X, Y), Q) \le \sqrt{C_{x,y}\mathcal{H}(Q|\mathcal{L}(X, Y))} \quad \text{for all } Q \in \mathcal{P}_2(\mathcal{C}([0,T],\mathbb{R}^{m_1 + m_2}))
	\end{equation*}
	where $\mathcal{H}$ is the Kullback-Leibler divergence defined\footnote{We use the convention $E[X] := +\infty$ whenever $E[X^+] = +\infty$.}, for any two probability measures $Q_1$ and $Q_2$  as
	 \begin{equation*}
	 	\mathcal{H}(Q_2|Q_1):= \begin{cases}
	 		E_{Q_2}[\log(\frac{dQ_2}{dQ_1})] &\text{ if } Q_2\ll Q_1\\
	 		+\infty &\text{ else}.
	 	\end{cases}
	 \end{equation*}
	 If one additionally assumes 
	 \begin{enumerate}[label = (\textsc{B2}''), leftmargin = 30pt]
		\item $|g(x)|\le L_f$,  $|f_t(x, y, z)|\le L_f(1 + |y|+ |z|)$ and $|b_t(x,y)|\le L_f(1 + |y|)$ for all $t, x, y,z$,
		\label{b1second}
	\end{enumerate}
	then \eqref{eq:talagrand for xy} holds for every $T>0$.
\end{lemma}
\begin{proof}
	This lemma follows from a combination of results in \cite{Dje-Gui-Wu}.
	First notice that the continuity of the paths of $(X, Y)$ is clear.
	In addition, there is a deterministic $L_v$--Lipschitz continuous, $v:[0,T]\times \mathbb{R}^{m_1} \to \mathbb{R}^{m_2}$ such that $Y_t^{s,x} = v(t, X^{s,x}_t)$ $P$-a.s., where $(X^{s,x}, Y^{s,x}, Z^{s,x})$ is the solution of \eqref{eq:fbsde normal} with $X^{s,x}_s=x$.
	We justify below that $v$ is $L_v$--Lipschitz continuous and the constant $L_v$ does not depend on $(m_1,m_2,d)$.
	But see already that as a consequence, the process $X$ satisfies the SDE
	\begin{equation*}
		X_t^{s,x} = x + \int_s^t\tilde b(u, X_u^{s,x})\,du + \sigma (W_t - W_s)
	\end{equation*}
	where the drift $\tilde b (t,x):= b(t, x, v(t,x))$ is $L_f(1+L_v)$--Lipschitz continuous with respect to the second variable.
	Thus it follows by \cite[Theorem 5]{Pal12} (which extends the original work \cite{Dje-Gui-Wu}) that the law $\mathcal{L}(X)$ of $X$ satisfies $T_2(C_1)$ with constant $C_1 = 4|\sigma|^2Te^{4T(L_f^2L_v^2T +1)}$.
	Therefore, by \cite[Lemma 2.1]{Dje-Gui-Wu}, we can now deduce that the law $\cL(X, Y)$ satisfies $T_2(C_{x,y})$ with 
	\begin{equation}
	\label{eq:Talagrand.fbsde.constant}
		C_{x,y}:=C_1(1 + L_v)^2.
	\end{equation}
	In particular, $C_{x,y}$ does not depend on $m_1, m_2$ and $d$.

	To conclude the proof, it remains to justify that $L_v$ does not depend on the dimension.
	If $T\le c(L_f)$ is sufficiently small, then this follows by \cite[Corollary 1.4]{Delarue}.
	If $T$ is arbitrary and the condition \ref{b1second} is satisfied, then this follows from \cite[Theorem 4.12]{MR3752669} or (the proof of) \cite[Theorem 2.5]{Kup-Luo-Tang18}.
	In the latter reference, it is actually shown that $L_v \equiv K_5 := \sqrt{2L_f^2 + L_fT}e^{L_fT}$.
\end{proof}
\begin{proof}(of Theorem \ref{thm:concentration syst})
	By triangular inequality, we have
	\begin{align}
	\nonumber
		P\left(\mathcal{W}^2_2(\cL(X_t, Y_t), L^N(\X_t, \Y_t)) \ge \varepsilon \right) &\le P\left( \mathcal{W}^2_2(\cL(X_t, Y_t), L^N(\tilde \X_t,\tilde \Y_t))\ge \varepsilon/2 \right)\\
		\label{eq:estim for concen}
		&+	P\Big(\frac1N\sum_{i=1}^N|\tilde X^i_t - X^{i,N}_t|^2 + |\tilde Y^i_t - Y^{i,N}_t|^2\ge \varepsilon/2 \Big).
	\end{align}
	The first term on the right hand side is estimated as
	\begin{align*}
		P\left( \mathcal{W}^2_2(\cL(X_t, Y_t), L^N(\tilde \X_t,\tilde \Y_t))\ge \varepsilon/2 \right) &\le C(a_{N,\frac{\varepsilon}{2}} 1_{\{\varepsilon<2\}} + b_{N,k,\frac{\varepsilon}{2}}).
	\end{align*}
	This follows by \cite[Theorem 2]{Fou-Gui15} since, by Lemma \ref{lem: moment d+5 system}, the processes $Y$ and $X$ have moments of order $k>4$.
	On the other hand, by Markov's inequality, we have
	\begin{align*}
		&P\Big(\frac1N\sum_{i=1}^N|\tilde X^i_t - X^{i,N}_t|^2 + |\tilde Y^i_t - Y^{i,N}_t|^2 \ge \varepsilon/2 \Big)\\
		 &\qquad \le \frac{2}{\varepsilon}\frac1N\sum_{i=1}^NE|\tilde X^i_t - X^{i,N}_t|^2 + E|\tilde Y^i_t - Y^{i,N}_t|^2\le \frac{C}{\varepsilon}\Big(r_{N,q+\xdim, k} + r_{N,\xdim,k} \Big),
	\end{align*}
	where the second inequality follows by Theorem \ref{thm:mom bound syst}.
	Combine this with \eqref{eq:estim for concen} to get \eqref{eq:conent bound}.

	Let us now turn to the proof of the concentration estimate \eqref{eq:gozlan}.
	Recall that the i.i.d. copies $(\tilde X^1,\tilde Y^1, \tilde Z^1), \dots,$ $(\tilde X^N,\tilde Y^N, \tilde Z^N)$, of $(X,Y,Z)$ solve the FBSDE \eqref{eq:MkVFBSDE} with $W$ replaced by $W^i$.
	Thus, they satisfy the equation \eqref{eq:fbsde normal} with $W$ replaced by $W^i$
	with the $L_f$--Lipschitz-continuous functions $b_t, f_t$ and $g$ being defined respectively as $g(x):= G(x, \cL(X_T))$, $f_t(x,y,z):= F_t(x,y,z,\cL(X_t, Y_t))$ and $b_t(x,y):= B_t(x,y, \cL(X_t, Y_t))$.
	Therefore, it follows by Lemma \ref{lem:Talagrand for system} that the law $\cL(X^i, Y^i) = \cL(X, Y)$ satisfies $T_2(C)$.
	Thus, by \cite[Theorem 1.3]{Gozlan09} we obtain \eqref{eq:gozlan}.
\end{proof}

%% file: section4.tex
\section{Approximation of the mean field game} 
\label{sec:limits}
This section of the paper is dedicated to the proofs of Theorems \ref{thm:main limit} and \ref{thm:concen Nash} stated in Section \ref{sec:main results}.
We start by the proof of the convergence of Nash equilibria.

\subsection{Proofs of Theorem \ref{thm:main limit} and Theorem \ref{thm:main.limit.arbitrary.time} }
In this section we provide the proof of the convergence of the Nash-equilibrium of the $N$-player game with interaction through state and control to the extended mean-field game.
The proof relies on the Pontryagin maximum principles derived in Section \ref{sec:pontryagin}, along with the propagation of chaos type results of the previous section.

Recall notation of Sections \ref{sec:main results} and \ref{sec:pontryagin} and the solution $(Y^{i,j}, Z^{i,j,k})_{i,j,k=1,\dots, N}$ of the adjoint equation of the game given in Equation \eqref{eq:adjoint-fctN}.
We will consider the off-diagonal processes $Y^{i,j}$, $i\neq j$ and then the diagonal terms $Y^{i,i}$.
The next two auxiliary results show that the off-diagonal elements of $Y^{i,j}$ converge to zero.
\begin{lemma}
	\label{lem:integrability}
	Assume that the conditions \ref{a1}-\ref{a5} are satisfied.
	If either $T$ is small enough for \ref{M} holds with $K_b$ large enough, then the solution $(Y^{i,j}, Z^{i,j,k})_{i,j,k=1,\dots, N}$ of the adjoint equation \eqref{eq:adjoint-fctN} along with the processes $X^{i,\underline{\hat\alpha}}$ satisfy
	\begin{equation*}
		E\bigg[\frac1N\sum_{i=1}^N|X^{i,\underline {\hat\alpha}}_t|^2 \bigg] \le C_t\quad \text{and}\quad
		E\bigg[|Y^{i,i}_t|^2 + \sup_{t \in [0,T]}|X^{i,\underline{\hat\alpha}}_t|^2 +\sum_{k=1}^N\int_0^T|Z^{i,i,k}_t|^2\,dt\bigg] \le C
	\end{equation*}
	for two constants $C_t,C>0$ which do not depend on $i,j,N$.
\end{lemma}
\begin{proof}
	This follows from the fact that the functions $b$, $ \partial_xf$ and $\partial_\xi f$ are of linear growth, and the functions $\partial_xb$  and $\partial_\mu b,$ are bounded (see conditions \ref{a1} and \ref{a4}).
	In fact, recalling that the adjoint equation is given by \eqref{eq:adjoint eq full terminal}-\eqref{eq:adjoint eq full}, these properties imply
	\begin{align}
	\notag
		E\Big[\sup_{t\in [0,T]}|X^{i,\underline {\hat\alpha}}_t|^2\Big] &\le CE\bigg[ |x^i_0|^2+\int_0^T \Big( 1 + |X^{i,\underline{\hat\alpha}}_u|^2 + |Y^{i,i}_u|^2 + \frac1N\sum_{k=1}^N|X^{k,\underline {\hat\alpha}}_u|^2 + |Y^{k,k}_u|^2  \Big)\,du\bigg] \\\label{eq:inequa X1}
		&\quad+ E\Big[|\sigma|^2\sup_{t \in [0,T]}|W^i_t|^2  \Big],
	\end{align}
	notice that we also used the representation of $\hat\alpha^{i,N}$ as $\hat\alpha^{i,N}_t = \Lambda(t,X^{i,\underline {\hat\alpha}}_t, Y^{i,i}_t,L^{N}(\underline X^{ \underline{\hat\alpha}}_t), \zeta^{i,N}_t )$ given in \eqref{eq:def alpha star}, and the estimation \eqref{eq:estimate.for.zeta} of $\zeta^{i,N}$.
	Subsequently taking the average over $i$ above, the expectation, and then applying Gronwall's inequality leads to 
	\begin{align}
	\label{eq:bound.aver.x.lemma}
		E\bigg[\frac1N\sum_{k=1}^N|X^{k,\underline{\hat \alpha}}_t|^2 \bigg]& \le C\bigg(1+E[|x^i_0|^2] +E\bigg[\frac1N\sum_{k=1}^N\int_0^T|Y^{i,i}_u|\,du\bigg]+  |\sigma|^2E[|W^i_T|^2]\bigg).
	\end{align}

	Let us now turn to the bound of $Y^{i,j}$ and $Z^{i,j,k}$.
	By It\^o's formula applied to $|Y^{i,i}|^2$, linear growth of $\partial_xf$ and $\partial_\mu f$ and boundedness of $\partial_xb$ and $\partial_\mu b$ we have
	\begin{align}
		\label{eq:inequa Y1}
		&E \bigg[ |Y^{i,i}_t|^2 + \sum_{k=1}^N\int_t^T|Z^{i,i,k}_u|^2\,du  \bigg]
		\le 
		C\bigg( 1 + E\Big[ |X^{i,\underline {\hat\alpha}}_T|^2 \Big] + \frac{1}{N} \sum_{k=1}^N  E\Big[ |X^{k,\underline {\hat\alpha}}_T|^2 \Big] \bigg)
		\\
		\notag
		&\quad +CE\bigg[ \int_t^T  |X^{i,\underline {\hat\alpha}}_u|^2 + \frac1N\sum_{k=1}^N|X^{k,\underline {\hat\alpha}}_u|^2   \,du \bigg]\\\notag
		&\quad + CE\bigg[\int_t^T3|Y^{i,i}_u|^2 + \frac1N\sum_{k=1}^N|Y^{k,k}_u|^2\,du \bigg].
	\end{align}
	Averaging out and using \eqref{eq:bound.aver.x.lemma}, it follows that when $T$ is small enough we have $\frac1N\sum_{j=1}^NE[\int_0^T|Y_t^{j,j}|^2]\,dt]<\infty$.
	Therefore, plugging this back in \eqref{eq:inequa Y1} and \eqref{eq:inequa X1} yields the result.
	We thus
	 arrive at the claimed bound for $Y^{i,j}$ and $Z^{i,j,k}$.

	The case where \ref{M} is satisfied for $K_b$ large enough follows exactly as in the proof of Lemma \ref{lem: moment d+5 system}.
	We omit the proof to avoid repetitions.
\end{proof}

\begin{lemma}
\label{prop:cv-Yij-0}
	If the conditions \ref{a1}-\ref{a5} are satisfied, then for every $i, j$ such that $i\neq j$, and every $t \in [0,T]$, we have
	\begin{equation*}
		E\Big[|Y^{i,j}_t |^2\Big]\le CN^{-1}\quad \text{for every $N\ge 1$ and some $C>0$}.
	\end{equation*}
\end{lemma}
\begin{proof}
	Let $i$ be fixed. 
	For every $j$ such that $i\neq j$, the process $Y^{i,j}$ satisfies the equation
	\begin{align*}
		dY^{i,j}_t&= - \left( \frac{1}{N} \partial_\mu f\left(t,X^{i,\underline {\hat\alpha}}_t, \hat\alpha^i_t, L^{N}(\underline X_t^{\underline{\hat\alpha}},\underline{\hat\alpha}_t)  \right)(X^{j,\underline {\hat\alpha}}_t) \right) dt
	\\
	&
	\quad - \bigg( \partial_{x} b\Big(t,X^{j,\underline {\hat\alpha}}_t, \hat\alpha^j_t, L^{N}(\underline X_t^{\underline{\hat\alpha}},\underline{\hat\alpha}_t) \Big) Y^{i,j}_t\\
	&\qquad
	+\sum_{k=1}^N \frac{1}{N} \partial_\mu b\Big(t,X^{k,\underline {\hat\alpha}}_t, \hat\alpha^k_t, L^{N}(\underline X_t^{\underline{\hat\alpha}},\underline{\hat\alpha}_t)   \Big)(X^{j,\underline{\hat \alpha}}_t) Y^{i,k}_t  \bigg)\, dt  + \sum_{k=1}^N Z^{i,j,k}_t dW^{k}_t
	\end{align*}
	with
	\begin{equation}
	\label{eq:termin ij}
	Y^{i,j}_T 
	=  \frac{1}{N} \partial_\mu g\left(X^{i,\underline{\hat\alpha}}_T, L^{N}(\underline X^{\underline{\hat\alpha}}_T)  \right) (X^{j,\underline{\hat\alpha}}_T).
	\end{equation}
	We assume for simplicity that $i=1$, and in an effort to write the equations in a more compact form, we define the vectors
	\begin{equation*}
		Y^{-1}:= (Y^{1,2},\dots, Y^{1,N}), \quad A_t:= \left(\partial_\mu f\left(t,X^{1,\underline {\hat\alpha}}_t, \hat\alpha^1_t, L^{N}(\underline X_t^{\underline{\hat\alpha}},\underline{\hat\alpha}_t)  \right)(X^{j,\underline {\hat\alpha}}_t) \right)_{j=2,\dots, N},
	\end{equation*}
	as well as
	\begin{equation*}
		B_t := \Big(\partial_\mu b\left(t,X^{1,\underline {\hat\alpha}}_t, \hat\alpha^1_t, L^{N}(\underline X_t^{\underline{\hat\alpha}},\underline{\hat\alpha}_t)   \right)(X^{j,\underline{\hat\alpha}}_t)\Big)_{j=2,\dots, N},
	\end{equation*}
	and the matrices
	\begin{equation*} 
		C_t:= \left( \partial_\mu b\left(t,X^{m,\underline {\hat\alpha}}_t, \hat\alpha^m_t, L^{N}(\underline X_t^{\underline{\hat\alpha}},\underline{\hat\alpha}_t)   \right)(X^{j,\underline {\hat\alpha}}_t)\right)_{m,j=2,\dots, N}		
		\end{equation*}
		and
		\begin{equation*}
		D_t:= \mathrm{diag}\left(\partial_{x} b\left(t,X^{j,\underline {\hat\alpha}}_t, \hat\alpha^j_t, L^{N}(\underline X_t^{\underline{\hat\alpha}},\underline{\hat\alpha}_t) \right) \right)_{j=2,\dots, N}.
	\end{equation*}
	With this new set of notation, the vector $Y^{-1}$ satisfies the multidimensional BSDE
	\begin{align*}
		dY^{-1}_t = -\left( \frac 1N (A_t +B_tY^{1,1}_t) + \frac1N  C_tY^{-1}_t + D_tY^{-1}_t  \right)\,dt + \sum_{k=1}^NZ^{1,-1, k}_t\,dW^k_t
	\end{align*}
	with terminal condition \eqref{eq:termin ij} and
	with $Z^{1, -1, k}:= (Z^{1,2, k},\dots, Z^{1,N,k})$.
	Thus, by square integrability of  $Z^{1, -1, k}$, it follows that
	\begin{align*}
		Y_t^{-1} = E\left[ Y^{-1}_T + \int_t^T\left( \frac 1N (A_s +B_sY^{1,1}_s) + \frac1N C_s Y^{-1}_s + D_sY^{-1}_s  \right)\,ds \mid \mathcal{F}_t^N \right].
	\end{align*}
	Denoting by $|\cdot|_2$ the Euclidean norm on $(\mathbb{R}^\xdim)^{N-1}$, we obtain
	\begin{align*}
		|Y^{-1}_t|^2_2 &\le 2(T+1)E\bigg[|Y^{-1}_T|_2^2 + \int_t^T\frac{1}{N^2} \big(|A_s|^2_2 + |B_s|_2^2|Y^{1,1}_s|^2\big)\\
		&\qquad  + \frac{1}{N^2}|Y^{-1}_s|_2^2|C_s|^2_2 + L_f^2|Y^{-1}_s|_2^2\,ds \mid \mathcal{F}_t^N\bigg],
	\end{align*}
	where we used definition of $D$ and the fact that $\partial_xb$ is bounded by $L_f$.
	Therefore, it follows by Gronwall's inequality that
	\begin{align}
	\label{eq:conv i diff j 1}
		|Y^{-1}_t|_2^2 \le CE\left[|Y^{-1}_T|^2_2 + \int_t^T\frac{1}{N^2} \big(|A_s|^2_2 + |B_s|_2^2|Y^{1,1}_s|^2\big)\,ds   \mid \mathcal{F}_t^N\right]
	\end{align}
	for a constant $C$ depending only on $T$ and the bound $L_f$ of $\partial_xb$ and $\partial_\mu b$, but not on $N$.
	In fact, since $\partial_\mu b$ is bounded by $L_f$, it follows that $|C_t|_2^2\le NL_f^2$.
	Moreover, since $\partial_\mu f$ is of linear growth (see assumption \ref{a4}), and $\hat{\alpha}^1$ is bounded in $\mathcal{H}^2(\mathbb{R}^m)$, it follows by Lemma \ref{lem:integrability} that the process $(A_t)$ is bounded in $\mathcal{H}^2(\mathbb{R})$ uniformly in $N$.	
	That is, it satisfies $\sup_NE\big[\int_0^T|A_t|^2\,dt\big]<\infty$.
	Since $\partial_\mu b$ is bounded by $L_f$ and by Lemma \ref{lem:integrability} $Y^{1,1}_t$ is bounded in $L^2$, it follows by Fubini's theorem that $E[\int_0^T|B_s|^2_2|Y_s^{1,1}|^2\,ds] \le NC$ for some constant $C>0$.
	In addition, it follows again by Lemma \ref{lem:integrability} that
	\begin{align*}
		E\big[|Y^{-1}_T|^2_2\big] &\le \frac{1}{N^2}E\Big[\sum_{j=2}^N|\partial_\mu g(X^{1,\underline{\hat\alpha}}_T, L^{N}(\underline X_T^{\underline{\hat\alpha}}) )(X^{j,\underline{\hat\alpha}}_T)|^2 \Big]\\
		&\le C\frac{N-1}{N^2}E\Big[|X^{1,\underline{\hat\alpha}}_T|^2 + \frac{1}{N}\sum_{k=1}^N|X_T^{k, \underline{\hat\alpha}}|^2 + 1\Big] +\frac{C}{N^2}E\Big[\sum_{j=2}^N|X_T^{j, \underline{\hat\alpha}}|^2\Big]\le \frac{C}{N}
	\end{align*}
	where the last inequality follows by Lemma \ref{lem:integrability}.
	Combine this with \eqref{eq:conv i diff j 1}, to obtain
	\begin{align*}
		E\Big[|Y^{-1}_t|^2_2 \Big] 
			&\le  C/N
	\end{align*}
	for some constant $C$ depending only on $T$, the bounds of $A, B$ and the second moment of $Y^{1,1}_s$ (which is bounded uniformly in $N$).
	Therefore, we have
	\begin{align*}
		E\Big[|Y^{1,j}_t|^2\Big]\le E\Big[|Y^{-1}_t|^2_2\Big]\le C\frac{1}{N}
	\end{align*}
	for some constant $C>0$ and for all $j=2,\dots, N$.
\end{proof}
Let us give a representation of the minimizer of the Hamiltonian.
\begin{lemma}
\label{lem:Lambda lipschitz}
	Assume that condition \ref{a2} holds. 
	Let $\Lambda :[0,T]\times \mathbb{R}^\xdim\times \mathbb{R}^m\times\mathcal{P}_2(\mathbb{R}^\xdim)\times \mathbb{R}^m\to \mathbb{R}^m$ be such that
	\begin{equation}
	\label{eq:lips alpha}
		\partial_af_1\big(t,x,\Lambda(t,x,y,\mu,\chi),\mu\big) + \partial_a b_1\big(t,x,\Lambda(t,x,y,\mu,\chi),\mu\big)y = \chi.
	\end{equation}
	Then $\Lambda$ minimizes the Hamiltonian $H$, is Lipschitz--continuous in $(x,y,\mu,\chi)$ with Lipschitz constant $L_\Lambda = \frac{L_f}{2\gamma}$ and satisfies the linear growth property
	\begin{equation}
	\label{eq:lin growth alpha}
	|\Lambda(t, x, y, \mu, \chi)|\le C\Big(1+ |x| + |y| + |\chi| + \big(\int_{\mathbb{R}^\xdim} |v|^2\mu(dv)\big)^{1/2} \Big).
	\end{equation}
\end{lemma}

\begin{proof}
	This lemma is probably well known but since we could not find a suitable reference, we provide the proof here for the sake of completeness.
	By convexity and differentiability of the Hamiltonian  (see \ref{a3}), a vector $\alpha = \Lambda(t, x,y,\mu,\chi)\in \mathbb{R}^m$ satisfying \eqref{eq:lips alpha} minimizes the function $\tilde H_1(t,x, a, \mu, y) := f_1(t, x, a, \mu) +b_1(t, x, a, \mu)y - \chi a$ in $a$.

	Let us show that $\Lambda$ is Lipschitz continuous.
	Let $(x,y,\mu,\chi),(x',y',\mu',\chi')$ be fixed put $\alpha' := \Lambda(t,x', y', \mu', \chi')$ and assume without loss of generality that $\alpha \neq \alpha^\prime$.
	By the condition \ref{a3}, (letting $\mu$, $\mu'$ be the first marginals of $\xi$ and $\xi'$ respectively) we have 
	\begin{align*}
		\gamma|\alpha - \alpha'|^2 &\le H(t, x, y, \alpha,\xi) - H(t, x, y, \alpha',\xi) - (\alpha - \alpha')\partial_aH(t, x, y, \alpha,\xi)\\
		&= \tilde H_1(t, x, y, \alpha,\mu) - \tilde H_1(t, x, y, \alpha',\mu) - (\alpha - \alpha')\partial_a\tilde H_1(t, x, y, \alpha,\mu)
	\end{align*}
	and
	\begin{align*}
		\gamma|\alpha - \alpha'|^2 \le \tilde H_1(t, x', y', \alpha',\mu') -\tilde  H_1(t, x', y', \alpha,\mu') - (\alpha' - \alpha)\partial_a\tilde H_1(t, x', y', \alpha',\mu').
	\end{align*}
	Summing up these two inequalities yields
	\begin{align*}
		2\gamma|\alpha - \alpha'|^2 &\le \int_0^1\partial_a\tilde H_1(t, x, y, u\alpha + (1-u)\alpha',\mu)\,du(\alpha-\alpha')\\
		&\quad  + \int_0^1\partial_a\tilde H_1(t, x', y', u\alpha + (1-u)\alpha',\mu')\,du(\alpha-\alpha')\\
		&\quad - (\alpha - \alpha')\big(\partial_a\tilde H_1(t, x, y, \alpha,\mu) - \partial_a\tilde H_1(t, x', y', \alpha', \mu') \big) \\
		&\le L_f|\alpha - \alpha'|\big( |x - x'| + |y-y'|+ \mathcal{W}_2(\mu, \mu') \big)
	\end{align*}
	for some constant $C>0$
	where we used Lipschitz continuity of $\partial_a\tilde H_1 = \partial_aH$ assumed in \ref{a3}.
	Therefore, we get
	\begin{equation*}
		|\alpha - \alpha'| \le C\big( |x - x'| + |y - y'| + |\chi - \chi'|+ \mathcal{W}_2(\mu, \mu') \big),
	\end{equation*}
	which shows that $\Lambda$ is Lipschitz continuous, therefore measurable.

	It remains to show the growth property.
	Assume without loss of generality that $\alpha \neq 0$.
	Using again \ref{a3}, we have
	\begin{align*}
		\gamma|\alpha|^2 &\le  H(t, x, y, 0, \xi) - H(t, x, y, \alpha, \xi) + \alpha \partial_aH(t, x, y, 0, \xi)\\
		&\le L_f|\alpha| + C|\alpha|\Big(1 + |x| + |y| + \big(\int |v|^2\mu(dv)\big)^{1/2} \Big)\\ 
		&\le C|\alpha|\Big( 1+|x| + |y| + \big(\int |v|^2\mu(dv)\big)^{1/2} \Big)
	\end{align*}
	for some constant $C$ where we used the linear growth condition on $\partial_aH$.
	Therefore, we have \eqref{eq:lin growth alpha}.
\end{proof}
We now come to the proof of the main result of the paper:
\begin{proof}(of Theorem \ref{thm:main limit})
	Let $\underline {\hat\alpha}$ be a Nash equilibrium for the $N$-player game. 
	By Theorem \ref{thm:N pontryagin}, the process $\underline{\hat\alpha}$ satisfies
	$\partial_{\alpha^i}H^{N, i}(t,\underline X_t, \underline {\hat\alpha}_t, \underline Y^{i,\cdot}_t) = 0$ for every $i = 1, \dots, N$.
	Unpacking this condition gives 	
	\begin{align}
	\notag
		&\partial_\alpha f_1\left(t, X_t^{i,\underline{\hat\alpha}}, \hat\alpha^i_t, L^{N}(\underline X^{\underline{\hat\alpha}}_t) \right) + \partial_\alpha b_1(t,X_t^{i, \underline{\hat\alpha}},\hat\alpha_t^i,L^{N}(\underline X^{ \underline{\hat\alpha}}_t))Y^{i,i}_t\\
		&\quad \frac1N\partial_\nu f_2(t,X_t^{i,\underline{\hat\alpha}}, L^{N}(\underline X^{\underline{\hat\alpha}}_t,\underline{\hat\alpha}_t) )(\hat\alpha^i_t) +\frac1N\sum_{k=1}^N\partial_\nu b_2(t,X_t^{i,\underline{\hat\alpha}}, L^{N}(\underline X^{\underline{\hat\alpha}}_t,\underline{\hat\alpha}_t) )(\hat\alpha^k_t)Y^{i,k}_t = 0. 
		\label{eq:first order}
	\end{align}
	This is due to the decompositions $b= b_1 + b_2$ and $f = f_1 + f_2$ and the fact that the functions $b_2$ and $f_2$ do not depend on $\hat\alpha^i$.
	By Lemma \ref{lem:Lambda lipschitz}, there is a Lipschitz continuous function $\Lambda$ such that
	\begin{equation}
	\label{eq:def alpha star}
		\hat\alpha^i_t = \Lambda\Big(t,X^{i,\underline {\hat\alpha}}_t, Y^{i,i}_t,L^{N}(\underline X^{ \underline{\hat\alpha}}_t), \zeta^N_t \Big)
	\end{equation}
	whereby
	\begin{equation*}
		\zeta^{i,N}_t := -\frac1N\partial_\nu f_2(t,X_t^{i,\underline{\hat\alpha}}, L^{N}(\underline X^{\underline{\hat\alpha}}_t,\underline{\hat\alpha}_t) )(\hat\alpha^i_t) -\frac1N\sum_{k=1}^N\partial_\nu b_2(t,X_t^{i,\underline{\hat\alpha}}, L^{N}(\underline X^{\underline{\hat\alpha}}_t,\underline{\hat\alpha}_t) )(\hat\alpha^k_t)Y^{i,k}_t
	\end{equation*}
	and $\Lambda$ not depending on $N$ and $i,j$ but only depending on $\partial_\alpha f_1$ and $\partial_\alpha b_1$.
	This shows that when $\underline{\hat\alpha}$ is a Nash equilibrium, then the optimal state $X^i \equiv X^{i, \underline {\hat\alpha}}$ along with the processes $(Y^{i,j},Z^{i,j,k})$ satisfy the fully coupled system of FBSDEs (recall \eqref{eq:adjoint-fctN})
	\begin{equation*}
	\begin{cases}
		dX^{i,\underline{\hat\alpha}}_t = b(t, X^{i, \underline{\hat\alpha}}_t,\hat\alpha_t^{i},L^N(\underline{X}^{\underline{\hat\alpha}}_t,\underline{\hat\alpha}_t) )\,dt +\sigma\,dW_t^i\\
		d Y^{i,i}_t = -\Big\{\partial_xf(t, X^{i, \underline{\hat\alpha}}_t, \hat\alpha^i_t, L^N(\underline{X}^{\underline{\hat\alpha}}_t,\underline{\hat\alpha}_t )) + \partial_xb(t, X^{i, \underline{\hat\alpha}}_t, \hat\alpha^i_t, L^N(\underline{X}^{\underline{\hat\alpha}}_t,\underline{\hat\alpha}_t ))Y^{i,i}_t+ \epsilon^N_t \Big\}\,dt\\
		\qquad + \sum_{k=1}^N Z^{i,j,k}_t dW^{k}_t   \\ 
		X^{i,\underline{\hat\alpha}}_0\sim \mu^{(0)},\,		\hat\alpha^i_t=\Lambda\Big(t,X^{i,\underline {\hat\alpha}}_t, Y^{i,i}_t,L^{N}(\underline X^{ \underline{\hat\alpha}}_t), \zeta^N_t \Big),\, Y^{i,i}_T = \partial_xg(X^{i, \underline{\hat\alpha}}_T, L^N(\underline{X}^{\underline{\hat\alpha}}_T)) +\gamma^N
	\end{cases}
	\end{equation*}
	with 
	\begin{align*}
		\varepsilon^{i,N}_t &:= \frac1N\partial_\mu f(t, X^{i, \underline{\hat\alpha}}_t, \hat\alpha^i_t, L^N(\underline{X}^{\underline{\hat\alpha}}_t,\underline{\hat\alpha}_t ))(X^{i, \underline{\hat\alpha}}_t)\\
		&\qquad + \frac1N\sum_{j=1}^N\partial_\mu b(t, X^{j, \underline{\hat\alpha}}_t, \hat\alpha^j_t, L^N(\underline{X}^{\underline{\hat\alpha}}_t,\underline{\hat\alpha}_t ))(X^{i, \underline{\hat\alpha}}_t)Y^{i,j}_t
	\end{align*}
	and
	\begin{equation*}
		\gamma^{i,N} := \frac1N\partial_\mu g(X^{i, \underline{\hat\alpha}}_T, L^N(\underline{X}^{\underline{\hat\alpha}}_T))(X^{i, \underline{\hat\alpha}}_T).
	\end{equation*}
	Unfortunately, we cannot directly apply the propagation of chaos results for FBSDE developed in the previous section to the above equation.
	For this reason, we introduce the following auxiliary equation:
	\begin{equation}
	\label{eq:aux fbsde 1}
	\begin{cases}
		d\widetilde X^{i,N}_t = b\Big(t, \widetilde X^{i,N}_t,\widetilde\alpha_t^{i,N},L^N(\underline{\widetilde X}_t,\underline{\widetilde\alpha}_t) \Big)\,dt +\sigma\,dW_t^i\\
		d\widetilde Y^{i,N}_t = - \Big\{\partial_x f\left(t, \widetilde X_t^{i,N}, \widetilde\alpha^{i,N}_t, L^{N}(\underline {\widetilde X}_t) \right) + \partial_x b\Big(t,\widetilde X_t^{i,N},\widetilde \alpha_t^{i,N},L^{N}(\underline{\widetilde X}_t) \Big)\widetilde Y^{i}_t \Big\} dt\\
		\qquad + \sum_{k=1}^N \widetilde Z^{i,k,N}_t dW^{k}_t,\\
		\widetilde X^{i,N}_0\sim \mu^{(0)},\, \widetilde Y^{i,N}_T = \partial_{x^i}g(\widetilde X^{i,N}_T,L^N(\underline{\widetilde X}_T)),\,	\widetilde\alpha^{i,N}_t=\Lambda\Big(t,\widetilde X^{i,N}_t, \widetilde Y^{i,N}_t,L^{N}(\underline{\widetilde X}_t), 0 \Big)
	\end{cases}
	\end{equation}
	and further define the function $ \varphi :[0,T]\times\cP_2(\RR^\xdim\times \RR^\xdim) \to \cP_2(\RR^\xdim\times \RR^m)$ given by
	$$
		\varphi (t,\xi) := \xi\circ \big(id_\xdim, \Lambda(t, \cdot, \cdot, \mu,0) \big)^{-1}
	$$
	where $id_\xdim$ is the projection on $\RR^\xdim$ and $\mu$ is the first marginal of the probability measure $\xi$, so that $(id_\xdim, \Lambda(t, \cdot, \cdot, \mu,0))$ maps $ \RR^\ell \times \RR^\ell$ to $\RR^\ell \times \RR^m$.
	Then, equation \eqref{eq:aux fbsde 1} can be re-written as
			\begin{equation}
		\label{eq:auxiliary fbsde}
		\begin{cases}
			d\widetilde X^{i,N}_t =  B\Big(t, \widetilde X^{i,N}_t, \widetilde Y^{i,N}_t, L^{N}(\underline{\widetilde  X}_t, \underline{\widetilde  Y}_t) \Big)\,dt + \sigma\,dW^i_t\\
			d\widetilde Y^{i,N}_t = -   F\Big(t, \widetilde X^{i,N}_t, \widetilde Y^{i,N}_t, L^{N}(\underline{\widetilde  X}_t,\underline{\widetilde Y}_t) \Big)\,dt  + \sum_{k=1}^N\widetilde Z^{i,k,N}_t\,dW^k_t\\
			\widetilde X^{i,N}_0\sim \mu^{(0)}, \,\,\, \widetilde Y^{i,N}_T =   G(\widetilde X^{i,N}_T, L^{N}(\underline{\widetilde  X}_T))
		\end{cases}
	\end{equation}
	with 
	\begin{equation*}
		 B(t, x, y, \xi ) := b\big(t,x, \Lambda(t,x, y, \mu, 0), \varphi(t,\xi)\big) 
	\end{equation*}
	\begin{align*}
		 F(t, x, y, \xi) &:=  \partial_{x} f\big(t,x, \Lambda(t,x, y, \mu, 0), \varphi(t,\xi) \big)  +  \partial_{x} b\big(t,x,\Lambda(t,x, y, \mu, 0), \varphi(t,\xi)\big)y
	\end{align*}
	where $\mu$ is the first marginal of $\xi$ and
	\begin{equation*}
		G(x, \mu) = \partial_xg(x,\mu).
	\end{equation*}

	Let us now justify that the functions $ B$, $ F$ and $ G$ satisfy the conditions of Theorem \ref{thm:mom bound syst}.
	By assumptions \ref{a1}, \ref{a3} and Lipschitz--continuity of $\Lambda$, in order to prove Lipschitz--continuity of $B, F, G$ it suffices to show that for every $\xi, \xi' \in \cP_2(\RR^\xdim\times \RR^\xdim)$ it holds
	\begin{equation*}
		\cW_2\big( \varphi(t,\xi), \varphi(t,\xi') \big) \le C\big( \cW_2(\xi, \xi') + \cW_2(\mu, \mu') \big)
	\end{equation*}
	where $\mu, \mu'$ are the first marginals of $\xi$ and $\xi'$, respectively.
	In fact, using Kantorovich duality theorem, see \cite[Theorem 5.10]{Vil2} that
	\begin{align*}
		&\cW_2^2\big(\varphi(t,\xi), \varphi(t,\xi')\big)\\
		 & = \sup\Big(\int_{\RR^\xdim\times \RR^m} h_1(x,y)\varphi(t,\xi)(dx,dy) - \int_{\RR^\xdim\times \RR^m} h_2(x',y')\varphi(t,\xi')(dx', dy') \Big)\\
					& = \sup\Big(\int_{\RR^\xdim\times \RR^\xdim} h_1(x,\Lambda(t,x,y, \mu))\xi(dx,dy) - \int_{\RR^\xdim\times \RR^\xdim} h_2(x',\Lambda(t, x', y',\mu'))\xi'(dx', dy') \Big)
	\end{align*}
	with the supremum over the set of bounded continuous functions $h_1,h_2:\RR^\xdim\times\RR^m\to \RR$ such that $h_1(x,y) -h_2(x',y') \le |x-x'|^2 + |y - y'|^2$ for every $(x,y), (x', y')\in \mathbb{R}^\xdim\times \RR^m$, which, by Lipschitz--continuity of $\Lambda$ implies that we have the following bound: $h_1(x, \Lambda(t, x, y, \mu)) - h_2( x', \Lambda(t, x',y', \mu'))\le |x -x'|^2 + |\Lambda(t, x, y, \mu) - \Lambda(t, x',y', \mu')|^2\le C\big( |x- x'|^2 + |y-y'|^2 + \cW_2^2(\mu, \mu') \big)$.
	This shows that
	\begin{align*}
		\cW_2^2\big(\varphi(t,\xi), \varphi(t,\xi^\prime)\big) \le \sup\Big( \int_{\mathbb{R}^l\times\mathbb{R}^l} \tilde h_1(x,y)\xi(dx,dy) - \int_{\mathbb{R}^l\times\mathbb{R}^l}\tilde h_2(x', y')\xi'(dx', dy') \Big)
	\end{align*}
	with the supremum over functions $\tilde h_1,\tilde h_2$ such that $\tilde h_1(x,y) - \tilde h_2(x', y') \le C\big(|x-x'|^2 + |y - y'|^2 + \cW_2^2(\mu, \mu') \big)$.
	Hence, applying Kantorovich duality once again yields
	\begin{align*}
		\cW_2^2\big(\varphi(t,\xi), \varphi(t,\xi^\prime)\big) \le C\inf\iint |x-x'|^2 + |y-y'|^2 + \cW_2(\mu, \mu')\,d\pi
	\end{align*}
	with the infimum over probability measures $\pi$ with first and second marginals $\pi_1 = \xi$ and $\pi_2 = \xi'$.
	This yields the result by definition  of $\cW_2(\xi, \xi')$.
	Therefore, $B,F$ and $G$ are Lipschitz continuous.

	That $B,F$ and $G$ are of linear growth follows by \ref{a3} and Lemma \ref{lem:Lambda lipschitz}.
	Therefore, the functions $B, F$ and $ G$ satisfy \ref{b1}-\ref{b2} with a constant $L_f$ which does not depend on $N$.
	As a consequence, it follows from \cite[Theorem 4.2]{MR3752669} that the equation \eqref{eq:auxiliary fbsde} admits a unique solution if $T$ is small enough.
 
	Similarly, by \cite[Theorem 4.24]{MR3752669} the following McKean-Vlasov FBSDE admits a unique solution $(X, Y, Z) \in \mathcal{S}^{2}(\mathbb{R}^\xdim)\times \mathcal{S}^{2}(\mathbb{R}^\xdim)\times \mathcal{H}^{2}(\mathbb{R}^{\xdim\times d})$ when $T$ is small enough:
	\begin{equation}
	\label{eq:MkV.in.proof}
		\begin{cases}
			dX_t = B(t, X_t, Y_t, \mathcal{L}(X_t ,Y_t) )\,dt + \sigma\,dW^i_t\\
			dY_t = - F(t, X_t, Y_t, \mathcal{L}(X_t ,Y_t) )\,dt + Z_t\,dW^i_t\\
			X_0\sim \mu_0, \,\,\, Y_T = \partial_xg(X_T, \mu_{X_T}).
		\end{cases}
	\end{equation}
	Thus, it follows from Theorem \ref{thm:mom bound syst} that there is a constant $\delta>0$ such that if $T\le \delta$, then for all  $N\in\mathbb{N}$ we have
	\begin{align*}
		E\bigg[\sup_{t\in [0,T]}|X_t - \widetilde X^{i,N}_t |^2 \bigg]+ E\Big[|Y_t - \widetilde Y^{i,N}_t |^2 \Big]\le C(r_{N,m+\xdim, k} + r_{N,\xdim,k})
	\end{align*}
	for some constant $C>0$ which does not depend on $N$, and where $r_{N,\xdim,k}$ is defined in \eqref{eq:def_r-nmqp}.
	On the other hand, using Lipschitz continuity (and definitions) of $B, F$ and $G$, it can be checked using standard FBSDE estimates that if $T$ is small enough, we have
	\begin{equation}
	\label{eq:auxliary bound}
		E\bigg[\sup_{t\in [0,T]}|X^{i,\underline{\hat\alpha}}_t - \widetilde X^{i,N}_t |^2 \bigg]+ E\Big[|Y^{i,i}_t - \widetilde Y^{i,N}_t |^2 \Big]\le CE[K^N]
	\end{equation}
	with 
	\begin{equation}
	\label{eq:def.K}
		K^{i,N} := |\gamma^{i,N}|^2 + \int_0^T|\varepsilon^{i,N}_t|^2 + |\zeta^{i,N}_t|^2\,dt 
	\end{equation}
	for a constant $C$ that does not depend on $N$.
	Therefore, we obtain by triangular inequality that
	\begin{align}
	\label{eq:diff xi x}
		E\bigg[\sup_{t\in [0,T]}|X^{i,\underline{\hat\alpha}}_t -  X_t |^2 \bigg]+ E\Big[|Y^{i,i}_t -  Y_t |^2 \Big] \le C\Big(E[K^N] + r_{N,m+\xdim, k} + r_{N,\xdim,k}\Big).
	\end{align}
	Let us check that $K^N$ converges to zero in expectation at the rate $N^{-1}$.
	By definition of $\varepsilon^N$, linear growth of $\partial_\mu f$ and boundedness of $\partial_\mu b$, we have
	\begin{align*}
		&E\bigg[\int_0^T|\varepsilon_t^N|^2\,dt \bigg]\\ 
		&\le CE\bigg[\int_0^T\frac{1}{N^2}\Big(1 + |X^{i,\underline{\hat\alpha}}_u|^2 + |X^{i,\underline{\hat\alpha}}_u|^2 +\frac1N\sum_{j=1}^N|X^{j,\underline{\hat\alpha}}_u|^2  \Big) + \frac1N\sum_{j=1}^N|Y^{i,j}_u|^2\,du  \bigg].
	\end{align*}
	Thus, by Lemma \ref{lem:integrability} and Proposition \ref{prop:cv-Yij-0} it holds that
	\begin{equation*}
	 	E\bigg[\int_0^T|\varepsilon_t^N|^2\,dt \bigg] \le \frac{C}{N}.
	\end{equation*} 
	Similarly, using linear growth of $ \partial_\nu f$ and boundedness of $\partial_\nu b$ we also obtain
	\begin{align}
	\notag
	 	E\bigg[\int_0^T|\zeta^N_t|^2\,dt \bigg] &\le \frac{C}{N^2}E\bigg[\int_0^T1 + |X^{i,\hat{\underline\alpha}}_t| + \frac1N\sum_{j=1}^N|X^{j,\hat{\underline\alpha}}_t|^2 \,dt \bigg]\\\notag
	 	&\quad + \frac1N\sum_{i=1}^NE\bigg[ \int_0^T|Y^{i,j}_t|^2\,dt\bigg]\\
	 	\label{eq:estimate.for.zeta}
	 	&\le C/N.
	\end{align} 
		Since $\partial_\mu g$ is of linear growth, see assumption \ref{a4} we have
		\begin{align*}
			E\big[ |\gamma^{N}|^2 \big] &\le \frac{1}{N^2}E\Big[ |\partial_\mu g(X^{i,\underline{\hat\alpha}}_T, L^{N}(\underline X_T^{\underline\alpha}) )(X^j_T)|^2\Big]\\
			& \le \frac{C}{N^2}E\Big[1 + |X^{i,\underline\alpha}_T|^2 +|X^{j,\underline\alpha}_T|^2 + \frac1N\sum_{k=1}|X^{k,\underline{\hat\alpha}}_T|^2 \Big]\le  C/N^2,
	\end{align*}
	where the last inequality follows from Lemma \ref{lem:integrability}.
	These estimates allow to conclude that 
	\begin{equation}
	\label{eq:bound.K}
		E[K^{i,N}]\le CN^{-1}.
	\end{equation}

	Now, put $\hat\alpha_t := \Lambda(t, X_t, Y_t, \mathcal{L}(X_t),0)$.
	For ease of notation, we will omit the zero in the last component and simply write
	\begin{equation}
	\label{eq:rem MFE}
		\hat\alpha_t := \Lambda(t, X_t, Y_t, \mathcal{L}(X_t)).
	\end{equation}
	By Lipschitz continuity of $\Lambda$, it follows that
	\begin{align}
	\nonumber
		E[|\hat\alpha^i_t - \hat\alpha_t|^2]&=E\Big[\Big|\Lambda(t, X_t^{i,\underline{\hat\alpha}}, Y_t^i, L^{N}(X_t^{\underline{\hat\alpha}}),\zeta^N_t) -  \Lambda(t, X_t, Y_t, \mathcal{L}(X_t)) \Big|^2\Big]\\\nonumber
		 & \le  CE\Big[ |X^{i,\underline{\hat\alpha}}_t - X_t|^2 + |Y^{i,i}_t - Y_t|^2 + \cW_2(L^{N}(X_t^{\underline{\hat\alpha}}),\mathcal{L}(X_t)) + |\zeta^N_t|^2 \Big]\\\label{eq:lambdaLipschitz}
		&\le C \Big(r_{N,m+\xdim, k} + r_{N,\xdim,k} + E[K^N]\Big).
	\end{align}
	Therefore, since $E[K^N]\le CN^{-1}$, we have
	\begin{equation*}
		E[|\hat\alpha^i_t - \hat\alpha_t|^2] \le C (r_{N,m+\xdim, k} + r_{N,\xdim,k} ).
	\end{equation*}

	It remains to justify that $\hat\alpha$ is indeed the mean field equilibrium.
	We apply again Proposition \ref{prop:Pontryagin-MFG} to justify that $\hat\alpha$ is the mean field equilibrium, thus we first show that the mapping $t\mapsto \cL(X_t, \hat\alpha_t)$ is bounded and Borel measurable.
	The Borel measurability follows by Lipschitz continuity of $\Lambda$ since by definition of the Wasserstein distance it holds that
	\begin{align*}
		\cW_2^2\big(\cL(X_t, \hat\alpha_t), \cL(X_s,\hat\alpha_s) \big) & \le E\big[|X_t - X_s|^2 + |\hat\alpha_t  - \hat\alpha_s|^2 \big]\\
		& \le CE\big[|X_t - X_s|^2 + |Y_t  - Y_s|^2 \big]
	\end{align*}
	for all $s,t \in [0,T]$.
	The boundedness of the second moment follows by Lemma \ref{lem:Lambda lipschitz} and square integrability of solutions of the McKean-Vlasov equation (recall Lemma \ref{lem: moment d+5 system}).
	In fact, we have
	\begin{equation*}
		\sup_{t\in [0,T]}E[|X_t|^2 + |\hat\alpha_t|^2] \le C\Big(1 + \sup_{t\in [0,T]}E[|X_t|^2 + |Y_t|^2]\Big) \le C,
	\end{equation*}
	which proves the claim. 
	Now,
	notice that, written in terms of $b,f$ and $g$, the McKean-Vlasov  system \eqref{eq:MkV.in.proof} reads
	\begin{equation}
	\label{eq:MkV fbsde full}
		\begin{cases}
			dX_t = b\big(t, X_t, \hat\alpha_t, \mathcal{L}(X_t,\hat\alpha_t) \big) + \sigma\,dW_t^i\\
			dY_t = - \Big\{\partial_xf\big(t, X_t, \hat\alpha_t, \mathcal{L}(X_t, \hat\alpha_t) \big) + \partial_xb(t, X_t, \hat\alpha_t,\mathcal{L}(X_t,\hat\alpha_t))Y_t\Big\} \,dt + Z_t\,dW^i_t\\
			X_0\sim \mu_0, \,\,\, Y_T = \partial_xg(X_T, \mathcal{L}(X_T)), \,\,\, \hat\alpha_t = \Lambda(t, X_t, Y_t, \mathcal{L}(X_t)).
		\end{cases}
	\end{equation}	
	This is the adjoint equation \eqref{eq:adjoint process mfg} associated to the mean field game.
	Since the functions $x\mapsto g(x,\mu)$ and $(x,a)\mapsto H(t, X_t, a, \cL(X_t, \hat\alpha_t), Y_t):= f(t, x, a,\cL(X_t, \hat\alpha_t) ) + b(t, x, a, \cL(X_t, \hat\alpha_t))y$ are $P\otimes dt$-a.s. convex, and by Lemma \ref{lem:Lambda lipschitz} the process $\hat\alpha_t$ satisfies
	\begin{equation*}
		H(t, X_t, \hat\alpha_t, Y_t, \cL(X_t, \hat\alpha_t)) = \inf_{a\in \mathbb{A}}H(t, X_t, a, Y_t, \cL(X_t, \hat\alpha_t)).
	\end{equation*}
	Thus, it follows from Pontryagin's stochastic maximum principle, see Proposition \ref{prop:Pontryagin-MFG} that $\hat\alpha$ is a mean field equilibrium.
	This concludes the proof.
\end{proof}

We conclude this subsection with the proof of the convergence to mean field equilibria in the case where monotonicity properties are assumed.

\begin{proof}(of Theorem \ref{thm:main.limit.arbitrary.time})
	The proof of Theorem \ref{thm:main.limit.arbitrary.time} is similar to that of Theorem \ref{thm:main limit}, except for two points. 
	
	First, to get well-posedness of the equations \eqref{eq:auxiliary fbsde} and \eqref{eq:MkV.in.proof}, we use \cite{Peng-Wu99} and \cite{Ben-Yam-Zhang15}, respectively. (This is where the condition \eqref{eq:mon.con.b.H.g} in \ref{M} is needed.)
	
	Next, in the present case we rely on the abstract propagation of chaos result Theorem \ref{thm:chaos.T.arbitray} rather than Theorem \ref{thm:mom bound syst}.
	Notice however that, in the arguments of the proof of Theorem \ref{thm:main limit}, in addition to the application of Theorem \ref{thm:mom bound syst}, having a short enough time horizon $T$ was also needed to get the estimate \eqref{eq:auxliary bound}.
	Thus, if we prove an analogous estimate, the rest of the proof remains the same, with Theorem \ref{thm:chaos.T.arbitray} applied instead of Theorem \ref{thm:mom bound syst}.

	Here, we will show that
	\begin{equation}
	\label{eq:auxliary.T.arbitrary}_{}
		E\bigg[\int_0^T|X^{i,\underline{\hat\alpha}}_t - \widetilde X^{i,N}_t |^2 + |Y^{i,i}_t - \widetilde Y^{i,N}_t |^2 \, dt \bigg] \le C/N.
	\end{equation}
	This is the analogue of \eqref{eq:auxliary bound} in the previous proof. 
	The proof of this inequality follows the strategy of the proof of Theorem \ref{thm:chaos.T.arbitray}.
	To avoid repetitions we give only the main steps of the argument.
	Recall the notation $L_{b,x}, L_{b,a}, L_{b,\xi}$ of the Lipschitz constant of $b$ in its arguments $x, a$ and $\xi$, respectively.
	Since $\Lambda$ is $L_\Lambda$--Lipschitz with $L_\Lambda = L_f/2\gamma$, a quick inspection shows that
	\begin{align*}
	 	|b(t,x,\Lambda(t,x,y,\mu,\zeta),\xi) - b(t,x',\Lambda(t,x',y',\mu',\zeta'),\xi')| &\le L_{B,x}|x-x'| + L_{B,\xi}\cW_2(\xi,\xi')\\
	 	& + L_{B,y}(|y-y'| + |\zeta-\zeta'|) 
	 \end{align*} 
	 with $L_{B,x} := L_{b,x} + L_{b,a}L_\Lambda$; $L_{B,\xi}:= L_{b,\xi} + L_{b,a}L_\Lambda$ and $L_{B,y}:= 2L_{b,a}L_{\Lambda}$.

	We will use the shorthand notation $\Delta X^i:= X^{i,\underline{\alpha}} - \widetilde X^{i,N}$, $\Delta Y^i:= Y^{i,i} - \widetilde Y^{i,N}$ and $\Delta Z^{i,j}:= Z^{i,i,j} - \widetilde Z^{i,i,N}$.
	Applying It\^o's formula to $|\Delta X^i|^2$, it follows by the monotonicity property \eqref{eq:mon.con.b} and Lipschitz--continuity of $b$ and $\Lambda$ that for every $\varepsilon>0$,
	\begin{align}
	\notag
		|\Delta X^i_t|^2
		&\le 2\int_0^t\Big({\color{black}\frac{L_{B,y} + L_{B,\xi}}{2\varepsilon} + \frac{2L_{B,\xi} + L_{B,y}}{2}} -K_b \Big)|\Delta X^i_u|^2 + L_{B,\xi}\frac1N\sum_{j=1}^N|\Delta X^j_u|^2\,du\\
		\label{eq:mon.proof.x}
		&\quad  + \int_0^t\varepsilon L_{B,y}|\Delta Y^i_u|^2 + \varepsilon L_{B,\xi}\frac1N\sum_{j=1}^N|\Delta Y^j_u|^2 
		+ L_{B,y}|\zeta^{i,N}_u|^2 + L_{B,\xi}\frac1N\sum_{j=1}^N|\zeta^{j,N}_u|^2\,du.
	\end{align}
	Thus, this implies
	\begin{align}
	\label{eq:ave.deltax.tlarge.proof}
		\frac1N\sum_{j=1}^N|\Delta X^j_t|^2& \le 2e^{\delta(\varepsilon)T}\int_0^t \varepsilon (L_{B,y} + L_{B,\xi}) \frac1N\sum_{j=1}^N(|\Delta Y^j_u|^2 + |\zeta^{j,N}_u|^2) \,du
	\end{align}
	with
	\begin{equation*}
		\delta(\varepsilon) := {\color{black}\frac{L_{B,y} + L_{B,\xi}}{2\varepsilon} + \frac{2L_{B,\xi} + L_{B,y}}{2} +  L_{B,\xi}} -K_b.
	\end{equation*}
	On the other hand, for the backward processes we have
	\begin{align}
	\notag
		|\Delta Y_t^i|^2 & + E\bigg[\sum_{j=1}^N\int_t^T|\Delta Z^{i,j,N}_u |^2\,du \mid \cF^N_t\bigg]\\\notag
		& \le E\Big[2L_f^2\Big(|\Delta X^i_T|^2 + \frac1N\sum_{j=1}^N|\Delta X^j_T|^2\Big) + |\gamma^{i,N}|^2 \mid \cF^N_t\Big] \\\notag
		&\quad + L_fE\bigg[\int_t^T6|\Delta Y^i_u|^2 + |\Delta X^i_u|^2 + |\zeta^{i,N}_u|^2 + |\varepsilon^{i,N}_u|^2\\
		\label{eq:mon.proof.y}
		& \qquad + \frac1N\sum_{j=1}^N(|\Delta Y^j_u|^2 + |\Delta X^j_u|^2 + |\zeta^{j,N}_u|^2 + |\varepsilon^{j,N}_u|^2)\,du \mid \cF^N_t \bigg].
	\end{align}
	This implies that
	\begin{align*}
		\frac1N\sum_{j=1}^N|\Delta Y_t^j|^2 
		&\le e^{7L_fT}E\Big[4L_f^2 \frac1N\sum_{j=1}^N|\Delta X^j_T|^2 + |\gamma^{i,N}|^2 \mid \cF^N_t\Big] \\
		&\qquad + 2e^{7L_fT}L_fE\bigg[\int_t^T\frac1N\sum_{j=1}^N( |\Delta X^j_u|^2 + |\zeta^{j,N}_u|^2 + |\varepsilon^{j,N}_u|^2)\,du \mid \cF^N_t \bigg].
	\end{align*}		
	Therefore, integrating on both sides and using \eqref{eq:ave.deltax.tlarge.proof} yields
	\begin{align*}
		&\frac1N\sum_{j=1}^NE\bigg[\int_0^T|\Delta Y^j_t|^2\,dt\bigg]  \le \Gamma_{\varepsilon,T}\frac1N\sum_{j=1}^NE\bigg[\int_0^T|\Delta Y^j_t|^2\,dt\bigg]+\\
		&\quad \Big(\Gamma_{\varepsilon,T} + e^{7 L_fT}T (1+2L_f)\Big) 
		E\bigg[\frac1N\sum_{j=1}^N|\gamma^{j,N}|^2 + \int_0^T\frac1N\sum_{j=1}^N(|\zeta_u^{j,N}|^2 + |\varepsilon^{j,N}_u|^2)\,du \bigg]
	\end{align*}
	with
	\begin{equation*}
		\Gamma_{\varepsilon,T} := 16T \varepsilon e^{7L_fT}L_f^2e^{\delta(\varepsilon)T} .
	\end{equation*}
	Choosing $\varepsilon$ small enough and then $K_b$ large enough,
	that is, such that
	\begin{equation*}
		K_B\ge  {\color{black} \frac{L_{B,y} + L_{B,\xi}}{2\varepsilon} + \frac{2L_{B,\xi} + L_{B,y}}{2} +  L_{B,\xi} }
	\end{equation*}
	with $\varepsilon < (16T  e^{7L_fT}L_f^2)^{-1}$.
	We thus have
	\begin{align*}
		\frac1N\sum_{j=1}^NE\bigg[\int_0^T|\Delta Y^j_t|^2\,dt\bigg]  &\le
		 CE\bigg[\frac1N\sum_{j=1}^N|\gamma^{j,N}|^2 + \int_0^T\frac1N\sum_{j=1}^N|\zeta_u^{j,N}|^2 + |\varepsilon^{j,N}_u|^2\,du \bigg]\\
		 &\le \frac1N\sum_{j=1}^{N}E[K^{j,N}]\le C/N
	\end{align*}
	for a constant $C>0$, and where the latter inequality follows by \eqref{eq:bound.K}.
	With this bound at hand, we proceed as in the proof of Theorem \ref{thm:chaos.T.arbitray} to show \eqref{eq:auxliary.T.arbitrary}.
	In particular, we plug this back into \eqref{eq:mon.proof.x} and \eqref{eq:mon.proof.y}.
\end{proof}

\subsection{Proof of Theorem \ref{thm:concen Nash}}
The proof is based on the representation \eqref{eq:def alpha star} and the concentration inequalities proved in Section \ref{sec:concen inequ}.

To show the moment bound, we consider the solution of the auxiliary forward backward SDE \eqref{eq:auxiliary fbsde} introduced in the proof of Theorem \ref{thm:main limit} and denote as usual $(\underline{\widetilde X}, \underline{\widetilde Y}, \underline {\widetilde Z}) = (\widetilde X^{i,N}, \widetilde Y^{i,N},\widetilde Z^{i,i,N})_{i=1,\dots,N}$.
Put
\begin{equation*}
	\widetilde\alpha^{i,N}_t := \Lambda\big( t, \widetilde X_t^{i,N}, \widetilde Y^{i,N}_t, L^N(\underline{\widetilde X}_t) \big).
\end{equation*}
Then by the representation \eqref{eq:rem MFE} of the mean field equilibrium, we have $\cL(\alpha_t) = \psi(t,\cL(X_t, Y_t))$,  where $\psi$ is the function given by
\begin{equation*}
	\psi(t,\xi) = \xi\circ \Lambda(t, \cdot,\cdot,\xi_1)^{-1}
\end{equation*}
for all $\xi\in  \cP_2(\RR^\xdim\times\RR^\xdim)$ with $\xi_1$ the first marginal of $\xi$.
Similarly, we have $L^N(\underline{\widetilde \alpha}_t) = \psi(t,L^N(\underline{\widetilde X}_t, \underline{\widetilde Y}_t))$.
As argued in the proof of Theorem \ref{thm:main limit}, the function $\psi$ is Lipschitz continuous for the $2$-Wasserstein metric, as a consequence of Lipschitz continuity of $\Lambda$.
Therefore, we have
\begin{align*}
	E\Big[\cW_2\big(L^N(\underline{\hat\alpha}_t), \cL(\hat\alpha_t)\big)\Big] &\le E\Big[\cW_2\big(L^N(\underline{\hat\alpha}_t), L^N(\underline{\widetilde \alpha}_t)\big)\Big] + E\Big[\cW_2\big(L^N(\underline{\widetilde \alpha}_t),\cL(\hat\alpha_t) \big) \Big]\\
	&\le E\bigg[ \Big(\frac1N\sum_{i=1}^N|X^{i,\hat{\underline{\alpha}}}_t - \widetilde X^{i,N}_t|^2 + |Y^{i,i,N}_t - \widetilde Y^{i,N}_t|^2 + |\zeta^N_t|^2 \Big)^{1/2}\bigg]\\
	&\quad + E\Big[\cW_2\big(\psi(t, L^N(\underline{\widetilde X}_t, \underline{\widetilde Y}_t)) , \psi(t, \cL(X_t, Y_t)) \big) \Big]\\
	&\le CE[K^N + |\zeta^N_t|^2]^{1/2} +  CE\Big[\cW_2\big( L^N(\underline{\widetilde X}_t, \underline{\widetilde Y}_t), \cL(X_t, Y_t) \big)\Big].
\end{align*}
It was showed in the proof of Theorem \ref{thm:main limit} that $E[K^N + |\zeta^{i,N}_t|^2]\le CN^{-1}$, and since the coefficients $B, F$ and $G$ of the FBSDE \eqref{eq:auxiliary fbsde} are Lipschitz--continuous, it follows from Theorem \ref{thm:mom bound syst} that $E\big[\cW_2\big( L^N(\underline{\widetilde X}_t, \underline{\widetilde Y}_t), \cL(X_t, Y_t) \big)\big]\le C\left( r_{N,2\xdim,k} + r_{N,\xdim,k} \right)$ for all  $(t,N)\in[0,T]\times\mathbb{N}$.
Therefore, we get
\begin{equation}
\label{eq:mom.bound.Wasser}
	E\big[\cW_2(L^N(\underline{\hat\alpha}_t), \cL(\hat\alpha_t))\big] \le C(N^{-1} + r_{N,2\xdim,k} + r_{N,\xdim,k}),
\end{equation}
which yields the claimed moment bound.

We now turn to the proof of the deviation inequality.
Let $h: \mathbb{R}^{mN}\to \mathbb{R}$ be a $1$-Lipschitz function and put
\begin{equation*}
	\tilde h(\underline x,\underline y) :=  h\Big(\Lambda(x^i, y^{i,i}, L^{N}(\underline x),0)_{i=1,\dots, N} \Big).
\end{equation*}
Consider again the solution $(\underline{\widetilde X}, \underline{\widetilde Y}, \underline {\widetilde Z}) = (\widetilde X^{i,N}, \widetilde Y^{i,N},\widetilde Z^{i,i,N})_{i=1,\dots,N}$ of the auxiliary FBSDE \eqref{eq:auxiliary fbsde} introduced in the proof of Theorem \ref{thm:main limit}.
Then, we have by \eqref{eq:def alpha star}, Lipschitz--continuity of $\Lambda$ and Chebyshev's inequality that
\begin{align*}
	P\Big(h(\underline{\hat\alpha}_t) - E[h(\underline{\hat\alpha}_t)]\ge a \Big)& \le P\Big( h(\underline{\hat\alpha}_t) - \tilde h(\underline{\widetilde X}_t, \underline{\widetilde Y}_t)\ge a/3\Big) + P\Big(E[ \tilde h(\underline{\widetilde X}_t, \underline{\widetilde Y}_t) - h(\underline{\hat\alpha}_t)]\ge a/3 \Big)\\
	&\quad+ P\Big(\tilde h(\underline{\widetilde X}_t, \underline{\widehat Y}_t) - E[\tilde h(\underline{\widetilde X}_t, \underline{\widetilde Y}_t)] \ge a/3\Big)\\
	&\le \frac{C}{a^2}\sum_{i=1}^N E\Big[ | X^{i,\underline{\hat\alpha}}_t - \widetilde X^{i,N}_t |^2 + | Y^{i,i,\underline{\hat\alpha}}_t - \widetilde Y^{i,N}_t |^2   +|\zeta_t^{i,N}|^2 \Big]\\
	&\quad+ P\Big(\tilde h(\underline{\widetilde X}_t, \underline{\widetilde Y}_t) - E[\tilde h(\underline{\widetilde X}_t, \underline{\widetilde Y}_t)] \ge a/3\Big)\\
		&\le \frac{C}{a^2}NE\big[ K^N + |\zeta^{i,N}_t|^2 \big] + P\Big(\tilde h(\underline{\widetilde X}_t, \underline{\widetilde Y}_t) - E[\tilde h(\underline{\widetilde X}_t, \underline{\widetilde Y}_t)] \ge a/3\Big).
\end{align*}
We showed in the proof of Theorem \ref{thm:main limit} that $E[K^N + |\zeta^N_t|^2]\le CN^{-1}$.
It now remains to estimate the last term on the right hand side above.
This is done using arguments similar to those put forth in the proof of \cite[Theorem 7]{backward-chaos}.
In fact, on the probability space $(\Omega^N, \mathcal{F}^N, P^N)$, consider the following compact form of the FBSDE \eqref{eq:auxiliary fbsde}:
\begin{align*}
	\begin{cases}
		\underline {\widetilde X}_t = \underline x + \int_0^t\underline B(u, \underline {\widetilde X}_u, \underline {\widetilde Y}_u)\,du + \Sigma \underline W_t\\
		\underline{\widetilde Y}_t = \underline G(\underline{\widetilde X}_T) + \int_t^T\underline F(u, \underline{\widetilde X}_u, \underline{\widetilde Y}_u)\,du - \int_t^T\underline {\widetilde Z}_u\,d\underline W_u
	\end{cases}
\end{align*}
where we put
\begin{equation*}
	\underline B(t, \underline x, \underline y) := (B(t, x^i, y^i, L^N(\underline x, \underline y)))_{i=1, \dots, N}
\end{equation*}
and similarly define $\underline F$ and $\underline G$, and we put $\Sigma := \text{diag}(\sigma)$ and $\underline Z := \text{diag}(\underline Z^{i,\cdot},\dots, \underline Z^{N, \cdot})$.
Then, by Lemma \ref{lem:Talagrand for system}, if $T$ is small enough, then the law $\mathcal{L}(\underline {X}, \underline { Y})$ of  $(\underline {X}, \underline { Y})$ satisfies Talagrand's $T_2(C_{x,y})$ inequality with constant $C_{x,y}$ depending on $L_f, T$ and $|\sigma|$ given in the proof of Lemma \ref{lem:Talagrand for system}, see Equation \eqref{eq:Talagrand.fbsde.constant}.
Note in passing that the Lipschitz constant $L_f$ of $\underline B, \underline F, \underline G $ does not depend on $N$.
Therefore, it follows from \cite[Theorem 1.3]{Gozlan09} that there is a constant $K$ depending on $C_{x,y}$ and the Lipschitz constant of $\tilde h$ such that
\begin{equation*}
	P\Big(\tilde h(\widetilde {\underline X}_t, \widetilde {\underline Y}_t) - E[H(\widetilde {\underline X}_t, \widetilde {\underline Y}_t)]\ge a/3 \Big) \le e^{-Ka^2}.
\end{equation*}
The bound $P\big( |\hat\alpha^{i,N}_t| - E[|\hat\alpha^{i,N}|] \ge a \big) \le 2e^{-Ka^2}$ follows by taking $h$ to be the absolute value of the projection on the $i$-th component and $N\ge\frac{1}{a}e^{Ka^2}$.

To get \eqref{eq:concen Wasses}, first notice that the function $\underline x\mapsto \sqrt{N}\cW_2(L^N(\underline x), \mathcal{L}(\hat\alpha_t))$ is $1$-Lipschitz for the norm $\|\underline x\|_{2,N}:=(\sum_{i=1}^N|x_i|^2)^{1/2}$.
Thus,  we have
\begin{align*}
	&P\big(\cW_2(L^N(\underline{\hat\alpha}_t), \cL(\hat\alpha_t)) \big) \\
	&\quad \le P\big(\sqrt{N}\cW_2(L^N(\underline{\hat\alpha}_t), \cL(\hat\alpha_t)) - \sqrt{N}E\big[\cW_2(L^N(\underline{\hat\alpha}_t), \cL(\hat\alpha_t)) \big]\ge \sqrt{N}a/2 \big) \\
	&\qquad + P\big(E\big[\cW_2(L^N(\underline{\hat\alpha}_t), \cL(\hat\alpha_t)) \big]\ge a/2 \big)\\
	&\quad \le \frac{C}{a^2N^2} + e^{-KNa^2} + P\big(E\big[\cW_2(L^N(\underline{\hat\alpha}_t), \cL(\hat\alpha_t)) \big]\ge a/2 \big).
\end{align*}
By \eqref{eq:mom.bound.Wasser}, choosing $N$ large enough the last term on the right hand side vanishes.
This concludes the proof for $T$ small enough.

Under the additional condition \eqref{eq:bound.on.H.for.T.large}, the functions $\underline B, \underline F$ and $\underline G$ satisfy \ref{b1second}, thus the proof of the case $T$ arbitrary is the same, in view of the second part of Lemma \ref{lem:Talagrand for system} and \eqref{eq:auxliary.T.arbitrary}.
 Note that one needs to observe that if $h$ is Lipschitz--continuous, then so is the function $\omega\mapsto \int_0^Th(\omega(t))\,dt$.
\hfill $\Box$

\textbf{Acknowledgement.}
The authors thank Julio Backhoff, Daniel Lacker and Dylan Possama\"i for fruitful discussions and helpful comments.